\documentclass[11pt]{article}
\usepackage[margin=1in]{geometry}
\usepackage{amsmath,amssymb,amsthm,mathtools}
\usepackage{booktabs}
\usepackage{array}
\usepackage{enumitem}
\usepackage{microtype}
\usepackage{hyperref}
\usepackage{bm}
\usepackage{graphicx}
\usepackage{makecell} 
\usepackage{multirow} 
\usepackage[table]{xcolor}

\usepackage{xeCJK}

\hypersetup{colorlinks=true,linkcolor=blue!55!black,citecolor=blue!55!black,urlcolor=blue!55!black}
\setlist{itemsep=0.25em,topsep=0.4em}
\allowdisplaybreaks

\theoremstyle{plain}
\newtheorem{theorem}{Theorem}[section]
\newtheorem{lemma}[theorem]{Lemma}

\newtheorem{corollary}[theorem]{Corollary}
\theoremstyle{definition}
\newtheorem{definition}[theorem]{Definition}

\newcommand{\R}{\mathbb{R}}
\newcommand{\vct}[1]{\bm{#1}}
\newcommand{\B}{\mathbb{B}}
\newcommand{\E}{\mathbb{E}}

\newcommand{\dist}{\operatorname{dist}}

\newcommand{\prog}{\operatorname{prog}}
\newcommand{\supp}{\operatorname{supp}}

\newcommand{\eps}{\epsilon}

\newcommand{\Dy}{D_{\mathcal{Y}}}

\newcommand{\norm}[1]{\left\lVert #1\right\rVert}

\newcommand{\Del}{\Delta_\Phi}

\newcommand{\Pcal}{\mathcal P}
\newcommand{\Rcal}{\mathcal R}
\newcommand{\equalcontrib}{\textsuperscript{*}}
\newcommand{\corresponding}{\textsuperscript{\dag}}

\title{Lower Bounds for Nonconvex-Concave Minimax Optimization}

\author{\begin{tabular}{@{}ccc@{}}
\parbox{0.28\textwidth}{\centering
Qilong Wu\equalcontrib\\[2pt]
The Chinese University of Hong Kong, Shenzhen\\[2pt]
{\footnotesize\texttt{qilongwu@link.cuhk.edu.cn}}
}
&
\parbox{0.28\textwidth}{\centering
Zhihao Gu\equalcontrib\\[2pt]
The Pennsylvania State University\\
{\footnotesize\texttt{zbg5155@psu.edu}}
}
&
\parbox{0.28\textwidth}{\centering
Junchi Yang\corresponding\\[2pt]
The Chinese University of Hong Kong, Shenzhen\\[2pt]
{\footnotesize\texttt{yangjunchi@cuhk.edu.cn}}
}
\end{tabular}}
\date{}

\begin{document}
\maketitle

\begingroup
\renewcommand{\thefootnote}{\fnsymbol{footnote}}
\footnotetext[1]{Equal contribution.}
\footnotetext[2]{Corresponding author. Also affiliated with Shenzhen Research Institute of Big Data and 
Shenzhen International Center for Industrial and Applied Mathematics.}
\endgroup

\begin{abstract}

We study lower bounds on the first-order oracle complexity of smooth nonconvex--concave minimax optimization. We consider objectives \(f\) that are jointly \(L\)-smooth in the primal and dual variables \((\vct{x},\vct{y})\), concave in \(\vct{y}\), and whose primal value function
$
\Phi(\vct{x}) := \max_{\vct{y}\in\mathcal Y} f(\vct{x},\vct{y})
$
satisfies the initial-gap condition
\(\Phi(\vct{0})-\inf_{\vct{x}\in\mathcal X}\Phi(\vct{x})\le \Delta_\Phi\), with a bounded dual domain satisfying
\(\operatorname{diam}(\mathcal Y)\le D_{\mathcal Y}\).
We measure stationarity by the norm of the gradient of the Moreau envelope of \(\Phi+\iota_{\mathcal X}\) with parameter \(1/(2L)\).
We prove that any deterministic zero-respecting first-order algorithm requires
$
\Omega\!\left(L^2D_{\mathcal Y}\Delta_\Phi\,\epsilon^{-3}\right)
$
oracle evaluations to find an \(\epsilon\)-stationary point. Under an unbiased stochastic first-order oracle with bounded variance, any stochastic zero-respecting algorithm requires
$
\Omega\!\left(L^3D_{\mathcal Y}^2\Delta_\Phi\,\epsilon^{-6}\right)
$
oracle evaluations. The same lower bounds hold when \(\Delta_\Phi\) is replaced by the initial primal--dual gap \(\mathcal G_0\). These deterministic and stochastic lower bounds match the corresponding upper bounds of \cite{lin2020nearoptimal} and \cite{zhang2022sapdplus}, respectively, up to a logarithmic factor in the deterministic setting.

\end{abstract}

\section{Introduction}
Minimax optimization provides a common formulation for a broad range of problems in machine learning, statistics, and robust decision making. Such problems arise in applications including two-player zero-sum games~\cite{v1928theorie}, generative modeling~\cite{goodfellow2020generative}, reinforcement learning~\cite{pinto2017robust,zhang2021multi}, robust and adversarial machine learning~\cite{madry2018towards,sinha2018certifiable}, and 
robust optimization~\cite{ben2009robust}. Optimization of pointwise maxima of nonconvex losses can also be written as a minimax problem through a dual representation~\cite{zhao2024primaldual}. A canonical formulation is
\begin{equation}
    \min_{\vct{x}\in\mathcal X}\max_{\vct{y}\in\mathcal Y} f(\vct{x},\vct{y}),
    \label{eq:minimax-intro}
\end{equation}
where $\vct{x}$ denotes the model or decision variable and $\vct{y}$ represents an adversary, perturbation, or dual variable. A substantial part of the recent complexity theory has focused on the nonconvex--strongly-concave (NC--SC) setting, for which the first-order complexity is now relatively well understood: near-optimal algorithms and matching lower bounds, up to logarithmic factors, are known under standard smoothness assumptions \cite{lin2020nearoptimal,li2021complexity,zhang2021complexity}. 

The strong-concavity assumption, however, can be restrictive in applications. In many adversarial and robust formulations, the inner problem is naturally concave but has no intrinsic strongly concave curvature. This motivates the more general nonconvex--concave (NC--C) setting, where $f(\vct{x},\cdot)$ is assumed only concave and the dual geometry is typically controlled through a bounded feasible set $\mathcal Y$, for example with \(\operatorname{diam}(\mathcal Y)\le D_{\mathcal{Y}}\). Define the primal value function
\[
    \Phi(\vct{x}):=\max_{\vct{y}\in\mathcal Y} f(\vct{x},\vct{y}),
\]
and its extended value function
$\varphi:=\Phi+\iota_{\mathcal X}$. Under $L$-Lipschitz smoothness assumption on $f$, \(\varphi\) is \(L\)-weakly convex. A standard stationarity measure for NC--C minimax optimization is based on the Moreau envelope of \(\varphi\), defined by
\begin{equation}
    \varphi_\lambda(\vct{x})
    :=
    \min_{\vct{z}\in\mathbb{R}^{d_x}}
    \left\{
        \varphi(\vct{z})
        +
        \frac{1}{2\lambda}\|\vct{z}-\vct{x}\|^2
    \right\},
    \qquad
    0<\lambda<L^{-1}.
    \label{eq:moreau-intro}
\end{equation}
The envelope $\varphi_\lambda$ is continuously differentiable for such
$\lambda$, even when $\Phi$ itself is not. The literature has focused on finding an \(\epsilon\)-stationary point \(\vct{x}\) satisfying
\begin{equation} \|\nabla\varphi_\lambda(\vct{x})\| \leq \epsilon, \label{eq:stationarity-intro} \end{equation}
or the corresponding expected stationarity criterion in the stochastic setting
\cite{davis2019stochastic,thekumparampil2019efficient,lin2020nearoptimal,
yang2020catalyst,kong2021accelerated,zhang2022sapdplus}.

Despite substantial progress on algorithms for NC--C minimax optimization, considerably less is known about its fundamental oracle complexity. The best-known first-order methods achieve \(\widetilde O(\epsilon^{-3})\) gradient evaluations in the deterministic setting~\cite{lin2020nearoptimal} and \(O(\epsilon^{-6})\) stochastic gradient evaluations under a bounded-variance oracle~\cite{zhang2022sapdplus}. In contrast, general lower bounds for NC--C minimax optimization remain much less developed. Since the NC--C class contains smooth nonconvex minimization and NC--SC minimax optimization as special cases, lower bounds for these problems immediately carry over to the NC--C setting. However, such inherited lower bounds do not match the best-known NC--C upper bounds. Existing lower-bound results tailored specifically to NC--C problems are instead largely algorithm-specific: they establish tight complexity guarantees for particular update rules, such as GDA, OGDA, extragradient, and perturbed or smoothed variants of GDA~\cite{mahdavinia2022tight,li2026smoothing}, rather than for a broad class of zero-respecting first-order methods.

\begin{table}[t]
\centering
\caption{Upper and lower bounds for $\Phi$-Moreau stationarity under
standard normalization. The function classes and the primal-value-gap and
primal--dual-gap parameterizations are defined in Definition~\ref{def:ncc-classes}. Universal constants, logarithmic factors, and the small-$\epsilon$ regime are suppressed.}\label{tab:main}
\renewcommand{\arraystretch}{1.35}
\small
\resizebox{\textwidth}{!}{\begin{tabular}{llcccc}
\toprule
\textbf{Setting}
&\textbf{Function class}
&\textbf{Oracle class}
&\textbf{Result}
&\textbf{References}
&\textbf{Complexity}
\\
\midrule

\multirow{3}{*}{Deterministic NC--C}
&$\mathcal F_{\rm NCC}(L,\Delta_\Phi,D_{\mathcal{Y}})$
&first-order
&Upper bound
&Corollary A.8 in \cite{lin2020nearoptimal}
&$\tilde{\mathcal O}(L^2D_{\mathcal{Y}}\Delta_\Phi\epsilon^{-3})$
\\

&\cellcolor{gray!12}$\mathcal F_{\rm NCC}(L,\Delta_\Phi,D_{\mathcal{Y}})$
&\cellcolor{gray!12}first-order zero-respecting
&\cellcolor{gray!12}Lower bound
&\cellcolor{gray!12}This paper (Theorem~\ref{thm:main})
&\cellcolor{gray!12}$\Omega(L^2D_{\mathcal{Y}}\Delta_\Phi\epsilon^{-3})$
\\

&\cellcolor{gray!12}$\mathcal F_{\rm NCC}^{\rm pd}(L,\mathcal G_0,D_{\mathcal{Y}})$
&\cellcolor{gray!12}first-order zero-respecting
&\cellcolor{gray!12}Lower bound
&\cellcolor{gray!12}This paper (Corollary~\ref{cor:det-pd-gap})
&\cellcolor{gray!12}$\Omega(L^2D_{\mathcal{Y}}\mathcal G_0\epsilon^{-3})$
\\

\midrule

\multirow{3}{*}{Stochastic NC--C}
&$\mathcal F_{\rm NCC}^{\rm pd}(L,\mathcal G_0,D_{\mathcal{Y}})$
&stochastic first-order
&Upper bound
&Theorem 5 in \cite{zhang2022sapdplus}
&$\mathcal O(L^3D_{\mathcal{Y}}^2\mathcal G_0\epsilon^{-6})$
\\
&\cellcolor{gray!12}$\mathcal F_{\rm NCC}(L,\Delta_\Phi,D_{\mathcal{Y}})$
&\cellcolor{gray!12}stochastic first-order zero-respecting
&\cellcolor{gray!12}Lower bound
&\cellcolor{gray!12}This paper (Theorem~\ref{thm:stoch-zr})
&\cellcolor{gray!12}$\Omega(L^3D_{\mathcal{Y}}^2\Delta_\Phi\epsilon^{-6})$
\\
&\cellcolor{gray!12}$\mathcal F_{\rm NCC}^{\rm pd}(L,\mathcal G_0,D_{\mathcal{Y}})$
&\cellcolor{gray!12}stochastic first-order zero-respecting
&\cellcolor{gray!12}Lower bound
&\cellcolor{gray!12}This paper (Corollary~\ref{cor:g0-lower-bound})
&\cellcolor{gray!12}$\Omega(L^3D_{\mathcal{Y}}^2\mathcal G_0\epsilon^{-6})$
\\
\bottomrule
\end{tabular}}
\end{table}

\subsection{Our Contributions and Main Results}
We establish oracle lower bounds for zero-respecting first-order algorithms. These bounds match the best-known upper bounds \cite{lin2020nearoptimal, zhang2022sapdplus}, up to a logarithmic factor in the deterministic setting. Table~\ref{tab:main} summarizes the resulting complexity landscape. Our main contributions are as follows:  
\begin{itemize}

\item \textbf{Oracle lower bounds for NC--C minimax optimization.}
We establish oracle lower bounds of
\(\Omega(L^2D_{\mathcal Y}\epsilon^{-3})\) in the deterministic setting and
\(\Omega(L^3D_{\mathcal Y}^2\epsilon^{-6})\) in the stochastic setting. Both bounds hold under either of two standard initial-gap conditions: a bounded initial primal--dual gap or a bounded initial gap measured through the Moreau envelope of the primal value function. The deterministic lower bound matches the upper bound of~\cite{lin2020nearoptimal} up to logarithmic factors, while the stochastic lower bound tightly matches the upper bound of~\cite{zhang2022sapdplus}. These results characterize its intrinsic difficulty relative to smooth nonconvex minimization and NC--SC minimax optimization.

\item \textbf{A bounded zero-chain construction.}
The deterministic lower bound is based on a sequence of primal coordinates coupled through dual blocks: each successive primal coordinate can become nonzero only after first-order information has propagated through the intervening dual block. The construction is scaled so that the unconstrained dual maximizer lies in the interior of the prescribed bounded dual domain and the primal value function admits an exact representation. Moreover, the joint gradient is Lipschitz continuous with a constant independent of the chain dimensions, while the norm of the Moreau-envelope gradient remains bounded away from zero whenever the terminal primal coordinate is zero. Thus, the construction combines the standard nonconvex zero-chain obstruction with the
additional oracle cost required for first-order information to propagate
through the dual variables.

\item \textbf{A stochastic dual-gate mechanism.}
For the stochastic lower bound, we replace the quadratic difference
terms $(y_k-y_{k+1})^2$ along each dual block with clipped quadratic
terms. This modification preserves both the primal value function and
the dual maximizers, while uniformly bounding the gradient components
corresponding to $y_2^{(i)},\ldots,y_N^{(i)}$ in each dual block.
We apply Bernoulli randomization only to these coordinates, leaving
$y_1^{(i)}$ and all primal coordinates deterministic. After contracting
the deterministic transitions, the progress analysis reduces to a
probability-$p$ zero-chain over 
randomized dual
coordinates. This yields the stochastic lower bound without imposing
a uniform bound on the full dual gradient.

\end{itemize}

\subsection{Related Work}
\label{sec:related-work}
\paragraph{Lower bounds for smooth nonconvex optimization.}
Lower bounds for first-order optimization have a long history dating back to the information-based complexity framework of Nemirovski and Yudin~\cite{nemirovski1983problem}. In the noiseless convex setting, oracle complexity is well understood for a broad range of smooth~\cite{Nesterov2018} and nonsmooth~\cite{nemirovski1983problem} problems. A common approach is to construct high-dimensional instances for which each oracle call reveals only limited information about the solution. This idea appears in several forms, including resisting-oracle arguments, linear-span restrictions, and zero-respecting constructions. Zero-respecting constructions are particularly useful for nonconvex problems, where the hard instance can be organized as a chain whose coordinates must be revealed sequentially.

For smooth high-dimensional nonconvex optimization, Carmon et al.~\cite{carmon2020lowerI,carmon2021lowerII} establish essentially tight first-order lower bounds for finding stationary points. For an \(L\)-smooth function with initial function gap at most \(\Delta\), their results imply that $\Omega\!\left(L\Delta\epsilon^{-2}\right)$ first-order oracle evaluations are necessary in the worst case to find an \(\vct{x}\) satisfying \(\|\nabla F(\vct{x})\|\leq\epsilon\). This bound matches the complexity of gradient descent up to constants under the standard normalization. Their construction has a sequential zero-chain structure: before the algorithm reveals a given coordinate, the oracle provides no information about subsequent coordinates. This zero-respecting framework has become a standard tool for proving lower bounds in high-dimensional nonconvex optimization and is also the starting point of the hard instance developed in this paper. In stochastic nonconvex optimization, Arjevani et al.~\cite{arjevani2023lower} combine deterministic nonconvex zero-chains with stochastic oracle constructions to obtain sharp lower bounds under bounded-variance and mean-squared-smoothness assumptions. Under a bounded-variance stochastic gradient oracle, the worst-case complexity can scale as \(\Omega(\epsilon^{-4})\), whereas mean-squared smoothness leads to the characteristic \(\epsilon^{-3}\) dependence. These results show that the zero-chain framework applies to both deterministic and stochastic oracle models.

\paragraph{Lower bounds for minimax optimization.}
For convex--concave saddle-point problems, lower complexity bounds were established using primal--dual hard instances in which information propagates alternately between the two variable blocks~\cite{ouyang2021lower,zhang2022lowerSaddle}. 
For nonconvex--strongly-concave (NC--SC) minimax optimization, Zhang et al.~\cite{zhang2021complexity} established lower bounds of
$
\Omega\!\left(L\Delta\sqrt{\kappa}\,\epsilon^{-2}\right)
$
in the deterministic setting and
$
\Omega\!\left(L\Delta\sqrt{\kappa}\,\epsilon^{-2}
+\kappa^{1/3}\sigma^2\epsilon^{-4}\right)
$
in the stochastic setting. Although NC--SC problems form a subclass of NC--C problems, the lower bounds in~\cite{li2021complexity,zhang2021complexity} cannot be directly extended to the NC--C setting by simply letting the strong-concavity parameter \(\mu\) tend to zero. In particular, increasing the condition number alone does not ensure that the dual maximizers remain within a domain of prescribed diameter. Imposing such a restriction on the dual domain may alter the induced primal value function and thereby invalidate the gradient lower bound established for the original construction. Likewise, rescaling the instance to enforce the diameter constraint changes both the smoothness and stationarity scales, leading to a weaker complexity lower bound than the one established here.

\paragraph{Upper bounds for NC--C minimax optimization.}
We next review existing upper bounds for NC--C minimax optimization under the stationarity measure adopted in this paper, namely, the norm of the gradient of a Moreau envelope of the primal value function.
In the deterministic setting, existing first-order methods achieve an essentially \(\widetilde O(\epsilon^{-3})\) complexity under this stationarity criterion. This rate was first attained by Thekumparampil et al.~\cite{thekumparampil2019efficient} using an inexact proximal-point method. It was subsequently achieved by several other approaches, including accelerated inexact proximal-point methods~\cite{lin2020nearoptimal,kong2021accelerated}, Catalyst-type schemes~\cite{yang2020catalyst}, and primal--dual smoothing methods for max-structured nonconvex optimization~\cite{zhao2024primaldual}. In the stochastic setting, Zhang et al.~\cite{zhang2022sapdplus} proposed SAPD+, which achieves an \(O(\epsilon^{-6})\) stochastic oracle complexity under a bounded-variance stochastic gradient oracle. Accordingly, the relevant upper-bound benchmarks are \(\widetilde O(\epsilon^{-3})\) in the deterministic setting and \(O(\epsilon^{-6})\) in the stochastic setting.

\paragraph{Notation.}
For a vector $\vct{z}\in \mathbb R^d$, $z_i$ is its $i$th coordinate,
$\supp(\vct{z}):=\{i\in[d]:z_i\neq0\}$ is its support, and $\norm{\vct{z}}$ and $\norm{\vct{z}}_\infty$
are its Euclidean and $\ell_\infty$ norms.  For a matrix $M$,
$\norm{M}_{\rm op}$ is its Euclidean operator norm and $M^\top$ is its
transpose.  We write $\vct{e}_i$ for the $i$th standard basis vector, $\B_2^d$ for
the Euclidean unit ball in $\R^d$, and $\mathsf P_C$ for Euclidean projection
onto a nonempty closed convex set $C$. 
For a nonempty closed convex set $C$, its indicator function is denoted by $\iota_C$, where
\[
\iota_C(\vct{z}):=
\begin{cases}
0, & \vct{z}\in C,\\
+\infty, & \vct{z}\notin C,
\end{cases}
\]
and its normal cone is
$N_C(\vct{z}):=\{\vct{v}:\langle \vct{v}, \vct{u}-\vct{z}\rangle\le 0
\text{ for every }\vct{u}\in C\}$. For $f:\mathcal X\times\mathcal Y\to\R$, we write $\nabla f=(\nabla_{\vct{x}} f,\nabla_{\vct{y}} f)$ for the joint gradient. All unspecified norms are Euclidean.

\section{Problem Setup}\label{sec:setup}
In this section, we first formalize the problem class, the stationarity measure, and the oracle model. We then introduce the two components used in both the deterministic and stochastic hard instances: the tridiagonal path matrix \(B_{\alpha,N}\), which defines the dual chain, and the smooth functions \(\nu\) and \(r\), which are used to construct the primal relay.

\subsection{Function Class and Stationarity Measure}

For a $f:\mathcal X\times\mathcal Y\to\R$, define its primal value
function by
\[
 \Phi_f(\vct{x}):=\max_{\vct{y}\in\mathcal Y}f(\vct{x},\vct{y}),\qquad \vct{x}\in\mathcal X,
\]
and its extended-valued version by
$\varphi_f:=\Phi_f+\iota_{\mathcal X}$ on the ambient primal space. 

We first introduce the function classes for smooth nonconvex--concave minimax problems. The two classes differ in how the initial optimality gap is measured. The first bounds the initial gap of the primal value function, whereas the second bounds the primal--dual gap at the reference point \((0,0)\). The former is used in \cite{lin2020nearoptimal}, while the latter is used in \cite{zhang2022sapdplus}. 

\begin{definition}[NC--C function classes]\label{def:ncc-classes}
Given $L,D_{\mathcal{Y}}>0$, consider continuously differentiable functions
$f:\mathcal X\times\mathcal Y\to\mathbb R$, over all finite dimensions
$d_x,d_y\in\mathbb N$, where
$\mathcal X\subset\mathbb R^{d_x}$ and
$\mathcal Y\subset\mathbb R^{d_y}$ are nonempty closed convex sets with
$0\in\mathcal X$ and $0\in\mathcal Y$, satisfying:
\begin{enumerate}
    \item $f$ is jointly $L$-smooth on $\mathcal X\times\mathcal Y$:
    \[
    \|\nabla f(\vct{x},\vct{y})-\nabla f(\vct{x}',\vct{y}')\|
    \le
    L\|(\vct{x},\vct{y})-(\vct{x}',\vct{y}')\|,
    \qquad
    \forall (\vct{x},\vct{y}),(\vct{x}',\vct{y}')\in\mathcal X\times\mathcal Y.
    \]

    \item For every $\vct{x}\in\mathcal X$, $f(\vct{x},\cdot)$ is concave.

    \item The dual domain satisfies: $\operatorname{diam}(\mathcal Y)
    :=
    \sup_{\vct{y},\vct{y}'\in\mathcal Y}\|\vct{y}-\vct{y}'\|
    \le D_{\mathcal{Y}}$.
    \item The initial scale is parameterized in either of the following
    two ways:
    \begin{enumerate}[label=(\roman*)]
        \item For $\Delta_\Phi>0$,
        $\mathcal F_{\rm NCC}(L,\Delta_\Phi,D_{\mathcal{Y}})$ denotes the class
        of functions satisfying
        \[
        \Phi_f(\vct{0})-\inf_{\vct{x}\in\mathcal X}\Phi_f(\vct{x})
        \le \Delta_\Phi;
        \]

        \item For $\mathcal G_0>0$,
        $\mathcal F_{\rm NCC}^{\rm pd}(L,\mathcal G_0,D_{\mathcal{Y}})$ denotes the class
        of functions satisfying
        \[
        \max_{\vct{y}\in\mathcal Y}f(0,\vct{y})
        -
        \inf_{\vct{x}\in\mathcal X}f(\vct{x},0)
        \le \mathcal G_0.
        \]
    \end{enumerate}
\end{enumerate}
\end{definition}

No convexity is imposed on \(f(\cdot,\vct{y})\). Since \(\mathcal Y\) is closed and has finite diameter in a finite-dimensional space, it is compact.
The continuity of \(f\) therefore ensures that the maximum defining
\(\Phi_f\) is attained for every \(\vct{x}\in\mathcal X\). These two quantities satisfy
\[
 \Phi_f(\vct{0})-\inf_{\vct{x}\in\mathcal X}\Phi_f(\vct{x})
 \leq
 \max_{\vct{y}\in\mathcal Y}f(0,\vct{y})
 -
 \inf_{\vct{x}\in\mathcal X}f(\vct{x},0),
\]
because \(\Phi_f(\vct{x})\geq f(\vct{x},0)\) for every
\(\vct{x}\in\mathcal X\). Thus, a bound on the initial primal--dual gap
also bounds the initial primal value gap. We will present the lower bounds for both classes. 

For \(\lambda>0\) and \(\vct{x}\in\R^{d_x}\), define the constrained
proximal set by
\begin{equation}\label{eq:prox-def}
 \Pcal_\lambda(\vct{x})
 :=
 \arg\min_{\vct{p}\in\mathcal X}
 \left\{
 \Phi_f(\vct{p})
 +\frac{1}{2\lambda}\norm{\vct{p}-\vct{x}}^2
 \right\},
 \quad
 G_\lambda(\vct{x};\vct{p})
 :=
 \lambda^{-1}(\vct{x}-\vct{p}),
 \quad
 \vct{p}\in\Pcal_\lambda(\vct{x}).
\end{equation}
The initial-gap condition implies that \(\Phi_f\) is bounded below on
\(\mathcal X\). 
Equivalently, 
the set \(\Pcal_\lambda(\vct{x})\) is the proximal set of the extended-value function
\(\varphi_f\) at \(\vct{x}\) with parameter \(\lambda\) on \(\R^{d_x}\). Although \(f\) is smooth, the primal value function \(\Phi_f\) need not be
differentiable because the dual maximizer may be nonunique. Nevertheless,
joint \(L\)-smoothness of \(f\) implies that
\(f(\cdot,\vct{y})\) is \(L\)-weakly convex for every
\(\vct{y}\in\mathcal Y\). 
Consequently, \(\varphi_f\) is a proper,
lower semicontinuous, and \(L\)-weakly convex function. We next recall the
standard Moreau-envelope regularity result for this class of functions.
\begin{lemma}[Lemma 2.2 in~\cite{davis2019stochastic}]
\label{lem:weak-convexity}
Let $\varphi:\mathbb R^{d_x}\to\mathbb R\cup\{+\infty\}$ be proper, lower semicontinuous, and $L$-weakly convex. For every $0<\lambda<L^{-1}$ and every
$\vct{x}\in\R^{d_x}$, the proximal problem has a unique
minimizer, denoted by $\vct{p}_\lambda^\varphi(\vct{x})$:
\[
\arg\min_{\vct{p}\in\R^{d_x}}
\left\{
\varphi(\vct{p})
+\frac{1}{2\lambda}\norm{\vct{p}-\vct{x}}^2
\right\}
=
\{\vct{p}_\lambda^\varphi(\vct{x})\}.
\] 
Moreover, the Moreau envelope
\[
\varphi_\lambda(\vct{x})
:=
\min_{\vct{p}\in\R^{d_x}}
\left\{
\varphi(\vct{p})+\frac{1}{2\lambda}\norm{\vct{p}-\vct{x}}^2
\right\}
\]
is continuously differentiable, and
\begin{equation}\label{eq:moreau-gradient}
\nabla\varphi_\lambda(\vct{x})
=
\lambda^{-1}
\bigl(\vct{x}-\vct{p}_\lambda^\varphi(\vct{x})\bigr).
\end{equation}
\end{lemma}

Applying Lemma~\ref{lem:weak-convexity} to
$\varphi_f=\Phi_f+\iota_{\mathcal X}$, we write
$\vct{p}_\lambda(\vct{x})
:=\vct{p}_\lambda^{\varphi_f}(\vct{x})$.
For $0<\lambda<L^{-1}$, this gives
\[
\Pcal_\lambda(\vct{x})
=\{\vct{p}_\lambda(\vct{x})\},
\qquad
\nabla(\varphi_f)_\lambda(\vct{x})
=
G_\lambda
\bigl(\vct{x};\vct{p}_\lambda(\vct{x})\bigr).
\] 
For both function classes, we measure stationarity using the gradient of the Moreau envelope of the value function, as defined below. This is a standard stationarity measure for weakly convex optimization~\cite{davis2019stochastic,davis2019proximally} and nonconvex--concave minimax optimization~\cite{zhang2022sapdplus}. By Lemma~\ref{lem:weak-convexity}, the corresponding stationarity conditions can equivalently be expressed in terms of
$
G_{1/(2L)}\bigl(\vct{x};\vct{p}_{1/(2L)}(\vct{x})\bigr).
$

\begin{definition}\label{def:stationarity}
Let $f\in\mathcal F_{\rm NCC}(L,\Del,\Dy) \cup
\mathcal F_{\rm NCC}^{\rm pd}(L,\mathcal G_0,D_{\mathcal{Y}})$.  A point $\vct{x}\in\mathcal X$ is
$\eps$-stationary for the primal value function if
\[
 \norm{\nabla(\varphi_f)_{1/(2L)}(\vct{x})}\le\eps.
\]
A random output $\vct{x}$ is expected $\eps$-stationary if
\[
 \E\norm{\nabla(\varphi_f)_{1/(2L)}(\vct{x})}\le\eps.
\]
\end{definition}

\subsection{Oracle and Algorithm Classes}

We consider deterministic and stochastic first-order algorithms for
nonconvex--concave minimax optimization. We begin by specifying the
corresponding first-order oracle models.

\begin{definition}[Deterministic first-order oracle]
The deterministic first-order oracle of a differentiable function
$f:\mathcal X\times\mathcal Y\rightarrow\mathbb R$ is the mapping
\[
\mathcal O_f^{\rm fo}(\vct{x},\vct{y})
=
(f(\vct{x},\vct{y}),\nabla_{\vct{x}} f(\vct{x},\vct{y}),\nabla_{\vct{y}} f(\vct{x},\vct{y})),
\qquad (\vct{x},\vct{y})\in\mathcal X\times\mathcal Y .
\]
\end{definition}

\begin{definition}[Stochastic first-order oracle]
The stochastic first-order oracle of a differentiable function
$f:\mathcal X\times\mathcal Y\rightarrow\mathbb R$ returns
\[
\mathcal O_f^{\rm sfo}(\vct{x},\vct{y})
=
(f(\vct{x},\vct{y}),\widehat\nabla_{\vct{x}} f(\vct{x},\vct{y};\xi),\widehat\nabla_{\vct{y}} f(\vct{x},\vct{y};\xi)), \qquad (\vct{x},\vct{y})\in\mathcal X\times\mathcal Y,
\]
where $\xi$ is a random variable satisfying
\[
\mathbb E_\xi[
\widehat\nabla f(\vct{x},\vct{y};\xi)]=\nabla f(\vct{x},\vct{y}),\qquad \mathbb E_\xi
\left[
\|
\widehat\nabla f(\vct{x},\vct{y};\xi)
-
\nabla f(\vct{x},\vct{y})
\|^2
\right]
\leq \sigma^2.
\]
\end{definition}

We restrict attention to zero-respecting first-order algorithms. Informally,
such an algorithm can update only coordinates that have been revealed by
previous first-order oracle responses. This class includes many standard first-order methods, including
gradient descent--ascent, extragradient, and their standard
momentum and stochastic variants
\cite{carmon2020lowerI,li2021complexity}.

\begin{definition}[Zero-respecting first-order algorithm]\label{def:zero-respecting}
A first-order algorithm is zero-respecting if, for every admissible
function $f$ and every oracle transcript, the $(t+1)$-th iterate
$(\vct{x}^{t+1},\vct{y}^{t+1})$ satisfies
\[
\vct{x}^{t+1}
\in
\left\{
\mathsf P_{\mathcal X}(\vct{u}):
\operatorname{supp}(\vct{u})
\subseteq
\bigcup_{0\leq i\leq t}
\left(
\operatorname{supp}(\vct{x}^i)
\cup
\operatorname{supp}(\nabla_{\vct{x}} f(\vct{x}^i,\vct{y}^i))
\right)
\right\},
\]
and
\[
\vct{y}^{t+1}
\in
\left\{
\mathsf P_{\mathcal Y}(\vct{v}):
\operatorname{supp}(\vct{v})
\subseteq
\bigcup_{0\leq i\leq t}
\left(
\operatorname{supp}(\vct{y}^i)
\cup
\operatorname{supp}(\nabla_{\vct{y}} f(\vct{x}^i,\vct{y}^i))
\right)
\right\}.
\]
In the stochastic setting, the same condition is imposed pathwise with
$\nabla_{\vct{x}} f(\vct{x}^i,\vct{y}^i)$ and $\nabla_{\vct{y}} f(\vct{x}^i,\vct{y}^i)$ replaced by the realized
stochastic gradients $\widehat\nabla_{\vct{x}} f(\vct{x}^i,\vct{y}^i;\xi_i)$ and
$\widehat\nabla_{\vct{y}} f(\vct{x}^i,\vct{y}^i;\xi_i)$, respectively.
\end{definition}

Throughout, all algorithms are initialized at
\((\vct{x}^0,\vct{y}^0)=(0,0)\), and the final primal output is also required
to satisfy the same zero-respecting condition as the primal iterates. Since Euclidean projection onto each feasible set constructed in Section~\ref{sec:hard} preserves coordinate support, Definition~\ref{def:zero-respecting} implies that every query is supported only on coordinates revealed by previous gradient evaluations.
 In the stochastic setting, the zero-respecting
condition is imposed pathwise with respect to the realized stochastic
gradients.

\subsection{Zero-chain Framework}
Zero-chain constructions are a standard tool for establishing first-order
oracle lower bounds
\cite{carmon2020lowerI,carmon2021lowerII,arjevani2023lower,
li2021complexity,ji2025lower,chen2025condition}.
To describe the sequential revelation of coordinates, for
\(\vct{v}\in\R^d\) and \(\theta\ge0\), define
\[
 \prog_\theta(\vct{v})
 :=
 \max\Bigl(
 \{j\in\{1,\ldots,d\}:|v_j|>\theta\}\cup\{0\}
 \Bigr).
\]
Thus, \(\prog_\theta(\vct{v})\) is the largest index of a coordinate
whose magnitude exceeds \(\theta\), or zero if no such coordinate exists.
In particular, \(\prog_0(\vct{v})\) is the largest index in the support
of \(\vct{v}\), with \(\prog_0(0)=0\).

\begin{definition}[Zero-chain]
A differentiable function \(f:\mathcal X\subseteq\R^d\to\R\) is a
first-order zero-chain if, for every \(\vct{x}\in\mathcal X\),
\[
 \prog_0(\nabla f(\vct{x}))
 \le
 \min\{\prog_0(\vct{x})+1,d\}.
\]
\end{definition}

This condition ensures that a gradient evaluation reveals no coordinate
with index greater than \(\prog_0(\vct{x})+1\). The following stochastic
analogue additionally bounds the probability of revealing that next
coordinate.

\begin{definition}[Probability-\(p\) zero-chain]
Let \(p\in(0,1]\), and let \(f:\mathcal X\subseteq\R^d\to\R\) be a
differentiable function equipped with a stochastic first-order oracle $\mathcal O_f^{\rm sfo}(\vct{x})
 =
 \bigl(f(\vct{x}),\widehat\nabla f(\vct{x};\xi)\bigr)$. We call \(f\) a probability-\(p\) zero-chain with respect to this oracle
if, for every \(\vct{x}\in\mathcal X\) with
\(r:=\prog_0(\vct{x})<d\),
\[
 \mathbb P_\xi\!\left(
   \prog_0(\widehat\nabla f(\vct{x};\xi))=r+1
 \right)\le p,
 \qquad
 \mathbb P_\xi\!\left(
   \prog_0(\widehat\nabla f(\vct{x};\xi))>r+1
 \right)=0.
\]
Here, the probability is taken over the randomness of the oracle.
\end{definition}

Thus, the stochastic gradient has no nonzero coordinate with index
greater than \(r+1\) almost surely, and its \((r+1)\)st coordinate is
nonzero with probability at most \(p\).

When Euclidean projection onto each feasible set preserves
coordinate support, a zero-respecting algorithm initialized
at the origin can reveal at most one new coordinate per
oracle call on a deterministic zero-chain. After $t$ oracle calls, its next query and any admissible
output are therefore supported on the first $t$ coordinates. After \(t\) oracle calls, its next query and any admissible output are
therefore supported on the first \(t\) coordinates. To establish a
deterministic lower bound, it thus suffices to construct a zero-chain
in the target function class for which the stationarity measure exceeds
\(\epsilon\) at every feasible point \(\vct{x}\) satisfying
\(\prog_0(\vct{x})<T\). Any zero-respecting algorithm must then make at
least \(T\) oracle calls to obtain an \(\epsilon\)-stationary point.
For a probability-$p$ zero-chain, each oracle call reveals
a new coordinate with conditional probability at most $p$,
given the preceding oracle transcript and the current query.
Consequently, revealing the first $T$ coordinates requires
$\Omega(T/p)$ oracle calls in expectation.

\subsection{Component Functions}

We now introduce several functions that will serve as building blocks for the hard instance constructed later. We also establish some of their key properties, which provide intuition for how these functions are used in the construction.

\paragraph{Quadratic chain.}
Following \cite{Nesterov2018,carmon2021lowerII}, for
\(\alpha>0\) and an integer \(N\ge2\), define the tridiagonal matrix
\begin{equation}\label{eq:path-matrix}
 B_{\alpha,N}
 :=
 \begin{bmatrix}
 1+\alpha^2&-1&&&\\
 -1&2&-1&&\\
 &\ddots&\ddots&\ddots&\\
 &&-1&2&-1\\
 &&&-1&1
 \end{bmatrix}
 \in\R^{N\times N}.
\end{equation}
The associated quadratic function is
\begin{equation}\label{eq:path-energy}
 f_{\alpha,N}(\vct{x})
 :=
 \frac12\vct{x}^\top B_{\alpha,N}\vct{x}
 =
 \frac{\alpha^2}{2}x_1^2
 +\frac12\sum_{k=1}^{N-1}(x_k-x_{k+1})^2.
\end{equation}
Each difference term couples two neighboring coordinates, while the
term \(\alpha^2x_1^2/2\) ensures that \(B_{\alpha,N}\) is positive
definite. The tridiagonal structure implies that each gradient component
depends only on the corresponding coordinate and its immediate
neighbors. Our later construction includes the negative of this quadratic function in the dual variable, which is a concave term. Its structure restricts the propagation of first-order information along the dual path and allows the dual maximizer to be computed explicitly.

\paragraph{Scalar component functions.}
For nonconvex optimization, Carmon et al.~\cite{carmon2021lowerII}
construct zero-chains using the smooth functions
\begin{equation}
 \bar\Psi(t)
 :=
 \begin{cases}
 0, & t\le\frac12,\\
 \exp\left(1-\frac{1}{(2t-1)^2}\right), & t>\frac12,
 \end{cases}
 \qquad
 \bar\Phi(t)
 :=
 \sqrt e\int_{-\infty}^{t}e^{-s^2/2}\,ds.
\end{equation}
The function \(\bar\Psi\) vanishes below a fixed threshold, which
restricts when one coordinate can affect the gradient of the next.
In our minimax instance construction, we want the dual maximizers to remain within a prescribed bounded domain while retaining a stationarity lower bound for the value function that is independent of the chain length. A direct rescaling of an unconstrained zero-chain from~\cite{carmon2021lowerII} does not guarantee both properties. We therefore introduce two continuously differentiable functions:

\begin{equation}\label{eq:relay-maps}
 \nu(t)
 :=
 \begin{cases}
 0, & t\le0,\\
 3t^2-2t^3, & 0<t<1,\\
 1, & t\ge1,
 \end{cases}
 \qquad
 r(t)
 :=
 \begin{cases}
 t, & t\le0,\\
 t(1-t), & 0<t<1,\\
 1-t, & t\ge1.
 \end{cases}
\end{equation}

The function \(\nu\) is the cubic transition function used in previous
zero-chain constructions~\cite{pan2026nonconvexpl}. It takes values in
\([0,1]\), equals zero on \((-\infty,0]\), and equals one on
\([1,\infty)\). We use it to couple successive primal coordinates
without introducing a direct gradient contribution at a coordinate
equal to zero, since \(\nu'(0)=0\).

The function \(r\) has zeros at \(0\) and \(1\), and satisfies
\(r'(0)=1\). In the hard instance, this nonzero derivative allows an
auxiliary variable to produce a nonzero gradient at the next primal
coordinate even when that coordinate is zero. 
The following lemma summarizes the key properties of these two functions. Its proof is deferred to Appendix~\ref{relay-properties}.

\begin{lemma}\label{lemma:relay-properties}
The functions \(\nu\) and \(r\) are \(C^{1,1}\) and satisfy
\begin{equation}\label{eq:scalar-bounds}
 |\nu'|\le\frac32,\qquad |\nu''|\le6,
 \qquad |r'|\le1,\qquad |r''|\le2,
\end{equation}
where the second-derivative bounds hold almost everywhere. Moreover,
\begin{equation}\label{eq:origin-identities}
 \nu(0)=\nu'(0)=r(0)=0,\qquad r'(0)=1,
 \qquad \nu'(t)=6r(t),\quad 0<t<1.
\end{equation}
For every \(\kappa\in(0,1/8)\), if \(t\le2\kappa\), then
\[
 \nu(t)\le2|r(t)|,
 \qquad
 r'(t)\ge1-4\kappa,
\]
and if \(t\ge1-2\kappa\), then
\[
 1-\nu(t)\le2|r(t)|.
\]
\end{lemma}

The derivative bounds are used to establish smoothness of the hard
instance. The remaining inequalities relate the values and derivatives
of the component functions near the two thresholds \(0\) and \(1\);
they are used to prove the stationarity lower bound before the terminal
primal coordinate is revealed.

\paragraph{Normalized coupling between primal coordinates.}
Fix an integer \(T\ge2\). For
\(\vct{u}\in\R^T\), define
\(\vct{q}(\vct{u})=(q_1(\vct{u}),\ldots,q_{T-1}(\vct{u}))\) by
\begin{equation}\label{eq:q-rho}
 q_i(\vct{u})
 :=
 \nu(u_i)\bigl(1-\nu(u_{i+1})\bigr),
 \qquad i=1,\ldots,T-1.
\end{equation}
Each component depends only on two successive coordinates. In
particular, \(q_i(\vct{u})=0\) if \(u_i\le0\) or \(u_{i+1}\ge1\),
whereas \(q_i(\vct{u})=1\) if \(u_i\ge1\) and \(u_{i+1}\le0\).
Thus, \(\vct{q}\) identifies transitions from a coordinate above the
upper threshold to a subsequent coordinate below the lower threshold.
Such a transition is central to the stationarity argument when the
first primal coordinate is near or above \(1\) but the terminal
coordinate is still zero.

Although each component of \(\vct{q}\) is bounded by one, its Euclidean
norm can grow with \(T\). We therefore define
\[
 \vct{\rho}(\vct{u})
 :=
 \frac{\vct{q}(\vct{u})}
 {\sqrt{1+\norm{\vct{q}(\vct{u})}^2}},
\]
so that
\begin{equation}\label{eq:rho-bound}
 \norm{\vct{\rho}(\vct{u})}\le1,
 \qquad \vct{u}\in\R^T.
\end{equation}
This uniform bound allows \(\vct{\rho}\) to be coupled to auxiliary
primal variables whose feasible radius is independent of \(T\).
The normalization preserves the zero coordinates of \(\vct{q}\).
Moreover, if \(u_j=0\), then \(\nu'(u_j)=0\), so the \(j\)th column
of \(D\vct{q}(\vct{u})\), and hence of
\(D\vct{\rho}(\vct{u})\), vanishes. This property is used to verify
that the normalized coupling preserves the required zero-chain
structure.

The next lemma gives uniform bounds on the Jacobians of \(\vct{q}\)
and \(\vct{\rho}\), where \(D\) denotes the Jacobian. All constants
are independent of \(T\), which is essential for obtaining a
smoothness bound independent of the dimensions of the hard instance. The proof is given in Appendix \ref{app:relay-derivatives}.

\begin{lemma}\label{lemma:relay-derivatives}
For all \(\vct{u},\vct{v}\in\R^T\),
\[
 \norm{D\vct{q}(\vct{u})}_{\rm op}\le3,
 \qquad
 \norm{D\vct{q}(\vct{u})-D\vct{q}(\vct{v})}_{\rm op}
 \le17\norm{\vct{u}-\vct{v}},
\]
and
\[
 \norm{D\vct{\rho}(\vct{u})}_{\rm op}\le3,
 \qquad
 \norm{D\vct{\rho}(\vct{u})-D\vct{\rho}(\vct{v})}_{\rm op}
 \le71\norm{\vct{u}-\vct{v}}.
\]
\end{lemma}

\section{The Hard Instance}\label{sec:hard}

In this section, we first construct the deterministic hard instance and then establish the properties needed to prove both the deterministic and stochastic lower bounds.

\subsection{Construction of the Hard Instance}
\label{subsec: deterministic hard instance}

The hard instance is constructed as the sum of primal and dual components. The primal component depends only on the primal variables and forms a chain along the primal coordinates. The dual component depends on both the primal and dual variables; it forms a chain along the dual coordinates while coupling the dual part to the primal chain~\cite{li2021complexity,ji2025lower,chen2025condition}.

\paragraph{Primal components.} Fix $T,N\ge2$. The primal variables are $\tilde{\vct{u}}\in\mathbb R^T,\,\,
\tilde{\vct{a}},\tilde{\vct{b}}\in\mathbb R^{T-1}$, with feasible domain
\begin{equation}\label{eq:C0}
 C_0:=\R^T\times\left\{(\tilde{\vct{a}},\tilde{\vct{b}})\in\R^{T-1}\times\R^{T-1}:
          \|\tilde{\vct{a}}\|^2+\|\tilde{\vct{b}}\|^2\le R^2\right\},
\end{equation}
where $R>0$ is a universal constant. Using the relay function $r$ defined in \eqref{eq:relay-maps}, define $\bar{\vct{r}}(\tilde{\vct{u}}):=(r(\tilde{u}_2),\ldots,r(\tilde{u}_T))$. Then, the primal component function $\Psi_0: \mathbb R^T\times \mathbb R^{T-1} \times \mathbb R^{T-1} \rightarrow \mathbb R$ is defined as
\begin{align}\label{eq:Psi0}
 \Psi_0(\tilde{\vct{u}},\tilde{\vct{a}},\tilde{\vct{b}})
 :=& H(\tilde{\vct{u}})+\frac12\norm{\tilde{\vct{a}}-\vct{\rho}(\tilde{\vct{u}})}^2
 +\frac12\norm{\tilde{\vct{b}}-\bar{\vct{r}}(\tilde{\vct{u}})}^2 +\frac12\bigl(\|\tilde{\vct{a}}\|^2+\|\tilde{\vct{b}}\|^2\bigr),
\end{align}
where $c_\eta>0$ is a universal constant. The function $H: \mathbb R^T \rightarrow \mathbb R$ is defined as
\[
H(\tilde{\vct{u}})
:=-c_\eta\tilde u_1-c_\eta\sum_{j=2}^T\nu(\tilde u_j)
  +\frac{c_\eta}{2}\sum_{j=1}^T r(\tilde u_j)^2.
\]

As illustrated in Figure~\ref{fig:zero_chain}, the two quadratic coupling
terms connect successive primal coordinates through the auxiliary 
variables $\tilde u_i$. The term
\(\frac12\|\tilde{\vct{a}}-\vct{\rho}(\tilde{\vct{u}})\|^2\)
couples \(\tilde u_i\) to \(\tilde a_i\) through
\(\rho_i(\tilde{\vct{u}})\), while
\(\frac12\|\tilde{\vct{b}}-\bar{\vct{r}}(\tilde{\vct{u}})\|^2\)
couples \(\tilde b_i\) to \(\tilde u_{i+1}\). In particular, when
\(\tilde u_{i+1}=0\), the latter term contributes
\(-\tilde b_i\) to the partial derivative with respect to
\(\tilde u_{i+1}\), since \(r(0)=0\) and \(r'(0)=1\).
The shifted coordinates in the definition of
\(\bar{\vct{r}}(\tilde{\vct{u}})\) therefore ensure that each auxiliary
variable \(\tilde b_i\) is coupled to the next primal coordinate.

The function \(H(\tilde{\vct{u}})\) ensures that this sequential
dependence leads to a stationarity lower bound. Its linear term
\(-c_\eta\tilde u_1\) produces a nonzero gradient in the first primal
coordinate at the origin. The terms involving \(\nu\) and \(r^2\)
are used to show that, at a point with sufficiently small constrained
stationarity residual, each coordinate of $\tilde{\vct{u}}$ must lie either near
zero or above a threshold close to one. The first coordinate must
lie above this threshold. If the terminal coordinate remains zero,
there must therefore be a transition between the two regions.
The coupling terms rule out a small stationarity residual at such
a transition, yielding the required lower bound before the terminal
primal coordinate is revealed.

\paragraph{Dual components.}
The dual variable consists of $T-1$ path blocks $\vct{y}^{(i)}\in\R^N$. Set $\alpha:=N^{-1/2}$. The matrix $B_{\alpha,N}$ in
\eqref{eq:path-matrix} satisfies the positive-definite identity in \eqref{eq:path-energy}.
For $a,b\in\R$, the function $h_N: (\mathbb R \times \mathbb R)\times \mathbb R^{N} \rightarrow \mathbb R$ is defined as
\begin{equation}\label{eq:path-block}
 h_N(a,b;\vct{y}):=\ell_0\left[-\frac12\vct{y}^\top B_{\alpha,N}\vct{y}
 +\left(\alpha a \vct{e}_1-\frac{\alpha}{2}b \vct{e}_N\right)^\top \vct{y}
 -\frac{\alpha^2(N-1)}8b^2\right],
\end{equation}
where \(\ell_0>0\) is a scaling parameter to be specified in the
lower-bound theorems.

The tridiagonal matrix \(B_{\alpha,N}\) couples neighboring dual coordinates, while the linear terms couple \(a\) to \(y_1\) and \(b\) to \(y_N\). Thus, first-order information propagates from one
endpoint to the other through the intervening dual coordinates.
Together with the primal coupling terms in \(\Psi_0\), these dependencies give the zero-chain structure shown in
Figure~\ref{fig:zero_chain}: a zero-respecting algorithm must reveal
a block of \(N\) dual coordinates before it can proceed to the next primal coordinate.

\paragraph{Rescaling and coupling the two components.}  For a scaling parameter \(s>0\), define the primal and dual feasible
sets by
\begin{equation}\label{eq:X0Y0}
 X_0
 :=
 \R^T\times
 \left\{
 (\vct{a},\vct{b})\in\R^{T-1}\times\R^{T-1}:
 \norm{\vct{a}}^2+\norm{\vct{b}}^2\le(Rs)^2
 \right\},
 \qquad
 Y_0
 :=
 \frac{\Dy}{2}\mathbb B_2^{(T-1)N}.
\end{equation}
The variables \(\vct{u}\) are unconstrained, whereas the auxiliary
variables \((\vct{a},\vct{b})\) lie in a Euclidean ball of radius
\(Rs\). This bound controls the norms of the unconstrained dual
maximizers of the path blocks. With the parameter choices specified
below, their concatenation lies in the interior of \(Y_0\) for every
feasible primal point. Consequently, restricting the dual
maximization to \(Y_0\) does not change its value.

For \(\vct{x}=(\vct{u},\vct{a},\vct{b})\in X_0\), define the
scaled variables
\begin{equation}\label{eq:normalized-coordinates}
 \tilde{\vct{u}}:=\vct{u}/s,
 \qquad
 \tilde{\vct{a}}:=\vct{a}/s,
 \qquad
 \tilde{\vct{b}}:=\vct{b}/s.
\end{equation}
Then
\((\tilde{\vct{u}},\tilde{\vct{a}},\tilde{\vct{b}})\in C_0\),
so the scaled primal domain is independent of \(s\).
The parameters \(\ell_0\) and \(s\) will be chosen to obtain the
required smoothness and stationarity bounds while satisfying the
dual feasibility condition.

We now define the hard instance $f:
 \bigl(\R^T\times\R^{T-1}\times\R^{T-1}\bigr)
 \times\R^{(T-1)N}
 \to\R$ by
\begin{equation}\label{eq:payoff}
 f(\vct{x},\vct{y})
 :=
 \ell_0s^2
 \Psi_0(\tilde{\vct{u}},\tilde{\vct{a}},\tilde{\vct{b}})
 +
 \sum_{i=1}^{T-1}h_N(a_i,b_i;\vct{y}^{(i)}).
\end{equation}
The minimax problem is defined by restricting \(f\) to
\(X_0\times Y_0\), and its primal value function is
\begin{equation}\label{eq:value-function}
 \Phi(\vct{x})
 :=
 \max_{\vct{y}\in Y_0}f(\vct{x},\vct{y}),
 \qquad \vct{x}\in X_0.
\end{equation}
For this instance, \(\Pcal_\lambda\) and \(G_\lambda\) denote the
proximal set and scaled proximal displacement defined in
\eqref{eq:prox-def}, with primal value function \(\Phi\) and feasible
set \(\mathcal X=X_0\).

\paragraph{Coordinate ordering.}
We order the joint coordinates according to the dependencies shown
in Figure~\ref{fig:zero_chain}. Starting from \(u_1\), each stage
passes through \(a_i\), the dual coordinates
\(y_1^{(i)},\ldots,y_N^{(i)}\), and \(b_i\) before reaching
\(u_{i+1}\). The total number of coordinates is
\begin{equation}\label{eq:K}
 K=1+(T-1)(N+3).
\end{equation}
Let \(\vct{\pi}=(\pi_1,\ldots,\pi_K)\) denote this ordering of
the coordinate labels. Explicitly,
\[
 \pi_{1+(i-1)(N+3)}=u_i,
 \qquad i=1,\ldots,T,
\]
and, for \(i=1,\ldots,T-1\) and \(j=1,\ldots,N\),
\[
 \pi_{2+(i-1)(N+3)}=a_i,
 \qquad
 \pi_{2+(i-1)(N+3)+j}=y_j^{(i)},
 \qquad
 \pi_{N+3+(i-1)(N+3)}=b_i.
\]
The corresponding ordering is also used for the normalized primal
coordinates. Since \(s>0\), normalization preserves their support.

For a joint vector $\vct{v}$, let $[\vct{v}]_{\pi_r}$ denote
its coordinate at position $r$ in the ordering $\vct{\pi}$
illustrated in Figure~\ref{fig:zero_chain}.
For $\theta\ge0$, define
\begin{equation}\label{eq:chain-progress}
\prog^{\vct{\pi}}_\theta(\vct{v})
:=
\max\Bigl(
\{r\in\{1,\ldots,K\}:|[\vct{v}]_{\pi_r}|>\theta\}
\cup\{0\}
\Bigr).
\end{equation}
For a primal vector, we apply this definition after adjoining zero dual components, retaining the same joint coordinate ordering.

\begin{figure}[t]
    \centering
    \includegraphics[width=1.0\linewidth]{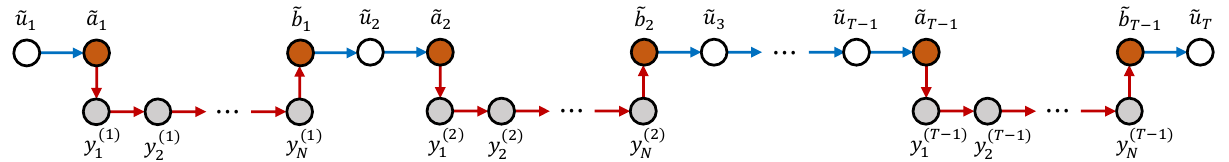}
    \caption{Zero-chain structure of the hard instance in \eqref{eq:payoff}. The primal coordinates $\tilde{u}_i,\tilde{a}_i,\tilde{b}_i$ are connected through dual path blocks $y^{(i)}_1,\ldots,y^{(i)}_N$, so information must propagate sequentially through each dual path before the next primal relay coordinate can be revealed.
    }
    \label{fig:zero_chain}
\end{figure}

\subsection{Structural Properties}

We next establish the properties of the hard instance needed for the
lower-bound analysis. These include an exact expression for the primal
value function and its initial gap, a smoothness bound independent of
the problem dimensions, and a stationarity lower bound when the terminal
primal coordinate is zero.

We first analyze the dual maximization problem and the resulting
primal value function. To state the value-function formula, define
$\Psi:\mathbb R^T\times\mathbb R^{T-1}\times
\mathbb R^{T-1}\to\mathbb R$ by
\begin{align}\label{eq:Psi}
\Psi(\tilde{\vct{u}},\tilde{\vct{a}},\tilde{\vct{b}})
:=
\Psi_0(\tilde{\vct{u}},\tilde{\vct{a}},\tilde{\vct{b}})
+\frac12
\norm{\tilde{\vct{a}}-\frac12\tilde{\vct{b}}}^2.
\end{align}
The following lemma shows that, under the stated feasibility
condition, the unconstrained dual maximizer is unique and lies
in the interior of $Y_0$ for every feasible primal point.
Thus, restricting the dual maximization to $Y_0$ leaves its
value unchanged.
The lemma expresses $\Phi$ as a scaled version of $\Psi$
and computes the exact initial value gap, which will be used
to choose parameters satisfying the gap bound of the target
function class.
The proof is given in Appendix~\ref{app:gap}.

\begin{lemma}\label{lem:gap}
Assume $2N^2R^2s^2\le(\Dy/4)^2$. 
Then the hard instance \eqref{eq:payoff} satisfies:
\begin{enumerate}[label=\arabic*.]
\item For every $\vct{x}\in X_0$, the function $f(\vct{x},\vct{y})$ is strictly concave in $\vct{y}$, and its unconstrained maximizer belongs to the interior of $Y_0$.
\item The constrained value function has the exact representation
\begin{equation}\label{eq:exact-value}
       \Phi(\vct{u},\vct{a},\vct{b})=\ell_0s^2\Psi(\vct{u}/s,\vct{a}/s,\vct{b}/s).
\end{equation}
\item The initial gap is
\begin{equation}\label{eq:exact-gap}
       \Phi(0)-\inf_{\vct{x}\in X_0}\Phi(\vct{x})
       =c_\eta\left(T+\frac12\right)\ell_0s^2.
\end{equation}
\end{enumerate}
\end{lemma}

We next verify joint smoothness of the hard instance. The following
lemma bounds the Lipschitz constant of \(\nabla f\) by \(C_{\ell}\ell_0\), independently of the chain lengths \(T,N\) and the scaling parameter \(s\). Thus, choosing \(\ell_0=L/C_{\ell}\)
ensures that \(f\) is jointly \(L\)-smooth, while leaving \(T,N\)
and \(s\) available to satisfy the remaining requirements of the
construction. The proof is given in Appendix~\ref{app:smooth}.

\begin{lemma}\label{lem:smooth}
For every $T,N\ge2$, every $s>0$, and every $\ell_0>0$, the hard instance \eqref{eq:payoff} satisfies
\begin{equation}\label{eq:smooth-base}
   \norm{\nabla f(\vct{z})-\nabla f(\vct{z}')}
   \le C_{\ell}\ell_0\norm{\vct{z}-\vct{z}'}
   \qquad\text{for all }\vct{z},\vct{z}'\in X_0\times Y_0,
\end{equation}
for a universal constant $C_{\ell}>0$.
\end{lemma}

The remaining step is to relate the zero-chain structure to
stationarity of the primal value function. Specifically, we show
that the norm of the Moreau-envelope gradient remains bounded away
from zero whenever the terminal primal coordinate is zero.
To establish this bound, we first analyze the normalized primal
function \(\Psi\) on \(C_0\). Define its constrained stationarity
residual by
\begin{equation}\label{eq:normal-residual}
 \Rcal_{C_0}(\tilde{\vct{u}},\tilde{\vct{a}},\tilde{\vct{b}})
 :=
 \dist\Bigl(
 0,\nabla\Psi(\tilde{\vct{u}},\tilde{\vct{a}},\tilde{\vct{b}})
 +N_{C_0}(\tilde{\vct{u}},\tilde{\vct{a}},\tilde{\vct{b}})
 \Bigr).
\end{equation}
This residual measures the violation of the first-order optimality
condition for minimizing \(\Psi\) over \(C_0\).

The following lemma first verifies that \(f\) is a first-order
zero-chain with respect to the ordering \(\vct{\pi}\) mentioned in Section \ref{subsec: deterministic hard instance}. It then
establishes a uniform positive lower bound on
\(\Rcal_{C_0}\) whenever \(\tilde u_T\le1/4\).
Using the optimality conditions for the proximal problem, this
bound yields a lower bound on the Moreau-envelope gradient norm
at every feasible primal point with \(u_T=0\).
Since \(u_T\) is the last coordinate in the ordering, the zero-chain
property ensures that the same bound applies to every deterministic
zero-respecting output produced in fewer than \(K\) oracle calls.
The proof is given in Appendix~\ref{app:zero-chain}.
\begin{lemma}\label{lem:gradient-lower}
Assume $2N^2R^2s^2\le(\Dy/4)^2$, and set
\begin{equation}\label{eq:Lbar}
       \bar L:=C_{\ell}\ell_0.
\end{equation}
Then the hard instance has the following properties.
\begin{enumerate}[label=\arabic*.]
\item The function \(f\) is a first-order zero-chain with respect
to the coordinate ordering \(\vct{\pi}\).
\item For every $(\tilde{\vct{u}},\tilde{\vct{a}},\tilde{\vct{b}})\in C_0$ with $\tilde{u}_T\le1/4$,
\begin{equation}\label{eq:normalized-obstruction}
       \Rcal_{C_0}(\tilde{\vct{u}},\tilde{\vct{a}},\tilde{\vct{b}})\ge C_\delta,
\end{equation}
for a universal constant $C_\delta>0$.
\item For every primal point $\vct{x}=(\vct{u},\vct{a},\vct{b})\in X_0$ with $u_T=0$ and every
$\vct{p}\in\Pcal_{1/(2\bar L)}(\vct{x})$,
\begin{equation}\label{eq:prox-lower-base}
       \norm{G_{1/(2\bar L)}(\vct{x};\vct{p})}
       >\frac{C_\delta}{2}\ell_0s.
\end{equation}
Consequently, the same bound holds for every deterministic zero-respecting
output produced in fewer than $K$ oracle calls.
\end{enumerate}
\end{lemma}

\section{Lower Bound for Deterministic NC-C Minimax Optimization}\label{sec:scaling}
We now derive the deterministic first-order oracle lower bound by
specifying the parameters of the hard instance in \eqref{eq:payoff}.
Using Lemmas~\ref{lem:gap}--\ref{lem:gradient-lower}, we choose
\(\ell_0,s,T,N\) so that the instance satisfies the prescribed
smoothness, initial-gap, and dual-diameter bounds, while the norm
of the Moreau-envelope gradient exceeds \(\eps\) at every feasible
primal point with \(u_T=0\). The zero-chain property then implies
that a zero-respecting algorithm must reveal the terminal primal
coordinate before returning an \(\eps\)-stationary point. Counting
the oracle calls required to reach this coordinate gives the
following theorem. Its proof is given in
Appendix~\ref{app:deterministic-theorem}.
\begin{theorem}[Deterministic zero-respecting lower bound]
\label{thm:main}
There exist universal constants \(c_0,c_1>0\) such that, for every
\(L,\Del,\Dy>0\) and every \(\eps\) satisfying
\begin{equation}\label{eq:small-eps}
 0<\eps\le
 c_1\min\{\sqrt{L\Del},\,L\Dy\},
\end{equation}
there is an instance
\(f\in\mathcal F_{\rm NCC}(L,\Del,\Dy)\) on finite-dimensional
closed convex sets \(X_0,Y_0\), with
\(\operatorname{diam}(Y_0)=\Dy\), for which the following lower
bound holds. Let $\Phi(\vct{x})
 :=
 \max_{\vct{y}\in Y_0}f(\vct{x},\vct{y})$ and $\varphi:=\Phi+\iota_{X_0}$. For any deterministic zero-respecting first-order algorithm,
every zero-respecting primal output \(\vct{x}\in X_0\) produced
after \(q\) joint oracle calls with
\begin{equation}\label{eq:q-lower}
 q<c_0L^2\Dy\Del\,\eps^{-3}
\end{equation}
satisfies
\begin{equation}\label{eq:eps-prox-lower}
 \norm{\nabla\varphi_{1/(2L)}(\vct{x})}
 =
 \norm{G_{1/(2L)}(\vct{x};\vct{p})}
 >
 \eps,
\end{equation}
where \(\vct{p}\) is the unique point in
\(\Pcal_{1/(2L)}(\vct{x})\).
Consequently, the deterministic zero-respecting first-order
oracle complexity of finding an \(\eps\)-stationary point is
\[
 \Omega\!\left(L^2\Dy\Del\,\eps^{-3}\right).
\]
\end{theorem}

Theorem~\ref{thm:main} expresses the lower bound in terms of the
initial primal value gap. For the hard instance constructed in
its proof, this gap equals the initial primal--dual gap.
Therefore, replacing \(\Del\) by \(\mathcal G_0\) in the same
construction gives the corresponding lower bound for
\(\mathcal F_{\rm NCC}^{\rm pd}(L,\mathcal G_0,\Dy)\).
The proof of the following corollary is given in
Appendix~\ref{app:det-pd-gap}.

\begin{corollary}\label{cor:det-pd-gap}
There exists a universal constant \(c>0\) such that, for every
\(L,\mathcal G_0,\Dy>0\) and every \(\eps\) satisfying
\[
 0<\eps\le
 c\min\{\sqrt{L\mathcal G_0},\,L\Dy\},
\]
there is an instance
\[
 f\in
 \mathcal F_{\rm NCC}^{\rm pd}(L,\mathcal G_0,\Dy)
\]
on which every deterministic zero-respecting first-order
algorithm requires
\[
 \Omega\!\left(L^2\Dy\mathcal G_0\,\eps^{-3}\right)
\]
joint oracle evaluations to return an \(\eps\)-stationary point
of the primal value function.
\end{corollary}

\section{Lower Bound for Stochastic NC-C Minimax Optimization}
\label{sec:stoch-zr}

We extend the deterministic construction to establish a lower bound
under unbiased stochastic first-order oracles with bounded variance.
The construction retains the same primal value function and the
stationarity lower bound at points with $u_T=0$. We modify the dual
blocks so that the gradient components corresponding to
$y_2^{(i)},\ldots,y_N^{(i)}$ are uniformly bounded, and apply
Bernoulli randomization only to these coordinates. The coordinates
$y_1^{(i)}$ and all primal coordinates are returned deterministically.

To reach the terminal primal coordinate, a zero-respecting algorithm
must cross $M=(T-1)(N-1)$ randomized dual coordinates. Each such coordinate requires a successful Bernoulli trial, while the intervening deterministic
transitions can be treated as free in the progress analysis.
Combining this progress bound with the stationarity obstruction yields
the stochastic oracle lower bound.

Throughout this section, we retain the notation and universal
constants introduced in Section~\ref{sec:hard}. In particular, we
use the same primal functions \(\Psi_0,\Psi\), feasible sets
\(C_0,X_0,Y_0\), coordinate ordering \(\vct{\pi}\), and chain length
\(K\). We take \(\alpha=N^{-1/2}\) and
\(\ell_0=L/C_\ell\), using the same universal constant
\(C_\ell\) fixed in Appendix~\ref{app:smooth}. Only the dual blocks are modified.

\begin{figure}[t]
    \centering
    \includegraphics[width=0.9\linewidth]{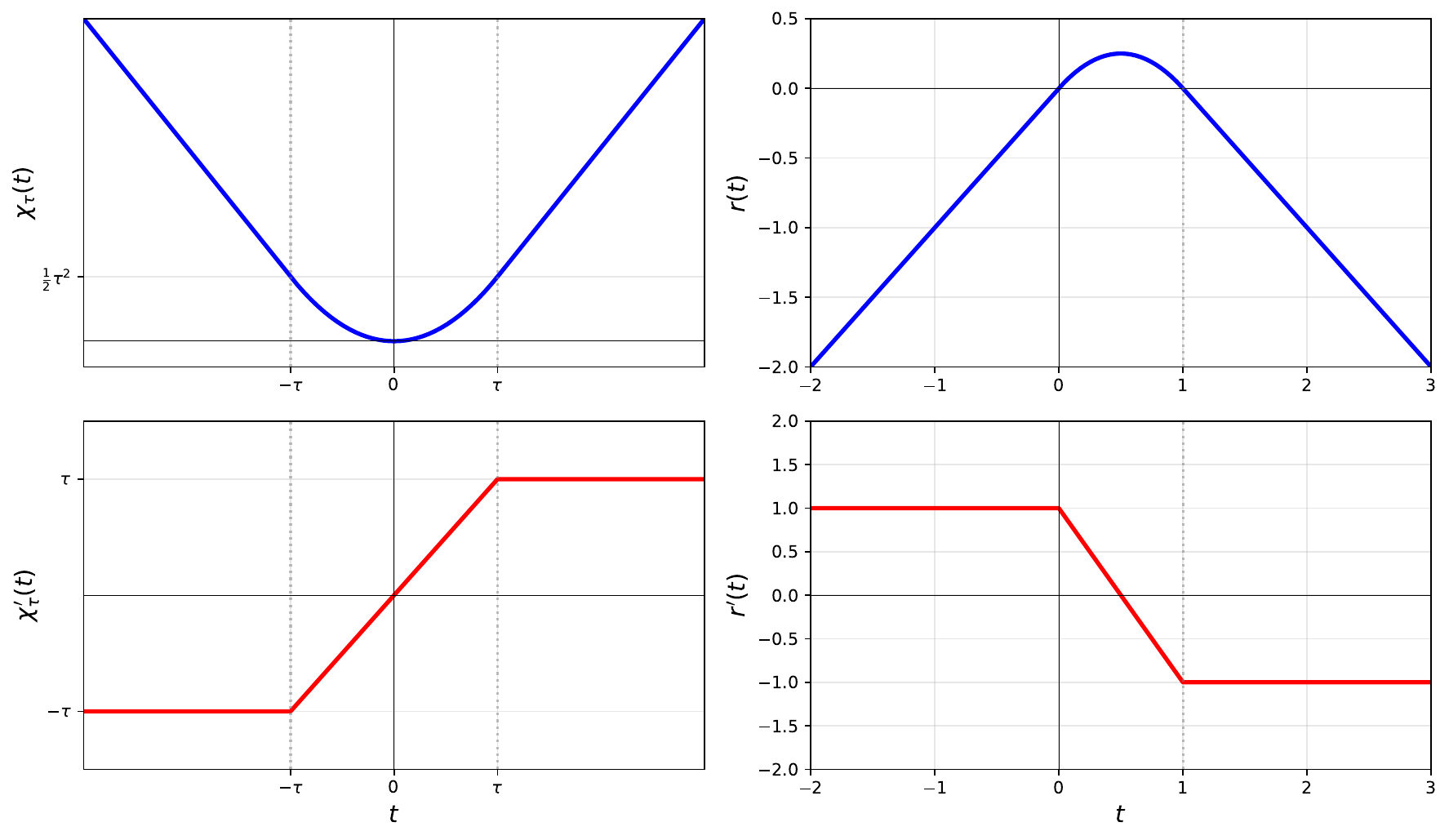}
    \caption{The functions $r$ and $\chi_\tau$, together with their derivatives. The function $r$ couples auxiliary primal variables to subsequent primal coordinates. The clipping in $\chi_\tau$ provides the uniform gradient bound needed for the randomized dual coordinates $y_2^{(i)},\ldots,y_N^{(i)}$.}
    \label{fig:auxiliary-functions}
\end{figure}
\subsection{Clipping the Dual Quadratic Terms}
The stochastic oracle will randomize only the dual coordinates $y_2^{(i)},\ldots,y_N^{(i)}$. To control the variance of this oracle, we need a uniform bound on the corresponding gradient components. In the deterministic construction, the quadratic difference terms produce gradients whose magnitude depends on the neighboring dual coordinates, and the diameter bound on $Y_0$ alone does not provide the required dependence on $N$. We therefore replace these terms with clipped quadratics.

For \(\tau>0\), define
\begin{equation}\label{eq:stoch-huber}
 \chi_\tau(t)
 :=
 \begin{cases}
  \frac12t^2, & |t|\le\tau,\\[2mm]
  \tau|t|-\frac12\tau^2, & |t|>\tau,
 \end{cases}
 \qquad
 \chi_\tau'(t)
 =
 \begin{cases}
  -\tau, & t<-\tau,\\
  t, & |t|\le\tau,\\
  \tau, & t>\tau.
 \end{cases}
\end{equation}
The function \(\chi_\tau\) is convex and continuously differentiable,
and its derivative satisfies
\begin{equation}\label{eq:stoch-huber-basic}
 |\chi_\tau'(t)|\le\tau,
 \qquad
 |\chi_\tau'(t)-\chi_\tau'(t')|\le|t-t'|,
 \qquad t,t'\in\R.
\end{equation}
The first bound controls the gradient contribution of each difference
term, while the second provides the regularity needed for a
smoothness bound independent of the problem dimensions.
The function and its derivative are shown in
Figure~\ref{fig:auxiliary-functions}.

We choose the clipping threshold as
\begin{equation}
 \tau_N:=R\alpha s.
\end{equation}
For the deterministic maximizer of a block with feasible endpoint
parameters \(a_i,b_i\), every difference between neighboring dual
coordinates equals \(\alpha b_i/2\). Since \(|b_i|\le Rs\),
its magnitude is at most \(\tau_N/2\). Thus, every such difference
lies strictly inside the quadratic region of \(\chi_{\tau_N}\).
The lemma below verifies that this choice preserves both the
maximizer and the maximum value of each feasible block.

For scalar parameters \(a,b\in\R\), we define the clipped dual block as
\begin{equation}\label{eq:stoch-hclip}
 h_N^{\rm clip}(a,b;\vct{y})
 :=
 \ell_0\left[
 -\sum_{k=1}^{N-1}\chi_{\tau_N}(y_k-y_{k+1})
 -\frac{\alpha^2}{2}y_1^2
 +\vct{c}(a,b)^\top\vct{y}
 -\frac{\alpha^2(N-1)}8b^2
 \right].
\end{equation}
where $\vct{c}(a,b):= \alpha a\vct{e}_1-\frac{\alpha}{2}b\vct{e}_N$. Replacing each deterministic dual block by this function gives
\begin{equation}\label{eq:stoch-payoff}
 f^{\rm sg}(\vct{x},\vct{y})
 :=
 \ell_0s^2
 \Psi_0(\tilde{\vct{u}},\tilde{\vct{a}},\tilde{\vct{b}})
 +
 \sum_{i=1}^{T-1}
 h_N^{\rm clip}(a_i,b_i;\vct{y}^{(i)}),
 \qquad
 \tilde{\vct{u}}=\frac{\vct{u}}s,\quad
 \tilde{\vct{a}}=\frac{\vct{a}}s,\quad
 \tilde{\vct{b}}=\frac{\vct{b}}s.
\end{equation}
We use the same feasible sets as in the deterministic construction:
\[
 X_0
 =
 \R^T\times
 \left\{
 (\vct{a},\vct{b})\in\R^{T-1}\times\R^{T-1}:
 \norm{\vct{a}}^2+\norm{\vct{b}}^2\le(Rs)^2
 \right\},
 \qquad
 Y_0
 =
 \frac{\Dy}{2}\mathbb B_2^{(T-1)N}.
\]
Define the primal value function by
\begin{equation}\label{eq:stoch-value}
 \Phi_{\rm sg}(\vct{x})
 :=
 \max_{\vct{y}\in Y_0}f^{\rm sg}(\vct{x},\vct{y}),
 \qquad \vct{x}\in X_0,
\end{equation}
and let
\(\varphi_{\rm sg}:=\Phi_{\rm sg}+\iota_{X_0}\)
denote its extended-valued version.

The following lemma establishes the properties needed for the stochastic construction. In addition to preserving the primal value function, clipping retains concavity, joint smoothness, and the zero-chain property. It also gives a uniform bound on the gradient components corresponding to $y_j^{(i)}$, $j=2,\ldots,N$, which are precisely the coordinates randomized by the stochastic oracle. The proof is given in Appendix~\ref{app:clipped-path}.

\begin{lemma}\label{lem:stoch-clipped-package}
Assume \(T,N\ge2\), \(\alpha=N^{-1/2}\), and
\begin{equation}\label{eq:stoch-dual-feas}
 2N^2R^2s^2\le(\Dy/4)^2.
\end{equation}
Then the instance \eqref{eq:stoch-payoff} has the following properties.
\begin{enumerate}[label=\arabic*.]

\item For every
\(\vct{x}=(\vct{u},\vct{a},\vct{b})\in X_0\)
and every \(i=1,\ldots,T-1\), the vector
\(\vct{y}^*(a_i,b_i)\) with components
\begin{equation}\label{eq:stoch-ystar}
 y_k^*(a_i,b_i)
 =
 \frac{a_i}{\alpha}
 -\frac{b_i}{2\alpha}
 -\frac{\alpha b_i}{2}(k-1),
 \qquad k=1,\ldots,N,
\end{equation}
is the unique maximizer of
\(h_N^{\rm clip}(a_i,b_i;\cdot)\) over \(\R^N\), and
\begin{equation}\label{eq:stoch-block-value}
 \max_{\vct{y}\in\R^N}
 h_N^{\rm clip}(a_i,b_i;\vct{y})
 =
 \frac{\ell_0}{2}\left(a_i-\frac{b_i}{2}\right)^2.
\end{equation}

\item The concatenation of these block maximizers lies in the
interior of \(Y_0\). Consequently, the constrained primal value
function equals that of the deterministic instance:
\begin{equation}\label{eq:stoch-value-identity}
 \Phi_{\rm sg}(\vct{u},\vct{a},\vct{b})
 =
 \ell_0s^2
 \Psi(\vct{u}/s,\vct{a}/s,\vct{b}/s).
\end{equation}
In particular,
\begin{equation}\label{eq:stoch-gap-exact}
 \Phi_{\rm sg}(0)
 -
 \inf_{\vct{x}\in X_0}\Phi_{\rm sg}(\vct{x})
 =
 c_\eta\left(T+\frac12\right)\ell_0s^2.
\end{equation}

\item For every \(\vct{x}\in X_0\), the function
\(f^{\rm sg}(\vct{x},\cdot)\) is concave on \(Y_0\).

\item The hard instance $f^{\rm sg}(\vct{x},\vct{y})$ is differentiable and satisfies
\begin{equation}\label{eq:stoch-smooth-package}
\norm{\nabla f^{\rm sg}(\vct{z})
-\nabla f^{\rm sg}(\vct{z}')}
\le
C_\ell\ell_0\norm{\vct{z}-\vct{z}'},
\qquad
\vct{z},\vct{z}'\in X_0\times Y_0.
\end{equation}
In particular, with $\ell_0=L/C_\ell$, the function $f^{\rm sg}$ is
$L$-smooth.

\item The function $f^{\rm sg}$ is a first-order zero-chain
with respect to the ordering $\vct{\pi}$:
\begin{equation}\label{eq:stoch-chain-progress}
\prog^{\vct{\pi}}_0\bigl(\nabla f^{\rm sg}(\vct{z})\bigr)
\le
\min\{\prog^{\vct{\pi}}_0(\vct{z})+1,K\},
\qquad
\vct{z}\in X_0\times Y_0.
\end{equation}
Moreover, for every $\vct{z}\in X_0\times Y_0$, we have \begin{equation}\label{eq:stoch-next-reveal}
\left|
\nabla_{y_j^{(i)}} f^{\rm sg}(\vct{z})
\right|
\le
G_N,
\qquad
i=1,\ldots,T-1,\quad j=2,\ldots,N.
\end{equation}
where $G_N=2\ell_0\tau_N$.

\end{enumerate}
\end{lemma}

\subsection{A Probability-$p$ Dual Zero-chain}
\label{subsec:stoch-oracle}
We first recall two standard results for stochastic zero-chains.
\begin{lemma}[Lemma 3 in~\cite{arjevani2023lower}]
\label{lem:bernoulli-zero-chain}
Let $f:\mathcal X\to\R$ be a zero-chain on
$\mathcal X\subset\R^T$. For $\vct{x}\in\mathcal X$, let $i^*(\vct{x}):=\inf\{i\in[T]:x_i=0\}$ denote the next coordinate to be discovered. For $p\in(0,1]$, define
\[
[\widehat{\nabla}f(\vct{x},\xi)]_i
:=
\begin{cases}
\dfrac{\xi}{p}\nabla_i f(\vct{x}),
& i=i^*(\vct{x}),\\[1ex]
\nabla_i f(\vct{x}),
& \text{otherwise},
\end{cases}
\]
where $\xi\sim\mathrm{Bernoulli}(p)$. Suppose
$\|\nabla f(\vct{x})\|_\infty\le G$ for every
$\vct{x}\in\mathcal X$. Then $O:\vct{x}\mapsto
\bigl(f(\vct{x}),\widehat{\nabla}f(\vct{x},\xi)\bigr)$ is a stochastic first-order oracle with variance bounded by $\sigma^2
\le
G^2\frac{1-p}{p}$.
Moreover, $f$ equipped with $O$ is a probability-$p$ zero-chain.
\end{lemma}

\begin{lemma}[Lemma 1 in~\cite{arjevani2023lower}]
\label{lem:probability-zero-chain-progress}
Let $f:\mathcal X\to\R$, where $\mathcal X\subset\R^T$ satisfies
$\operatorname{supp}(P_{\mathcal X}(\vct{x}))
=
\operatorname{supp}(\vct{x}), \,\,
\forall\,\vct{x}\in\R^T$, and suppose that $f$ is a probability-$p$ zero-chain equipped with a stochastic first-order oracle. Then, for any first-order algorithm, with probability at least $1-\delta$,
\[
x_T^t=0,
\qquad
\forall\,
t\le
\frac{T-\log(1/\delta)}{2p}.
\]
\end{lemma}

We now apply the same Bernoulli randomization to the clipped construction.
We randomize only the dual coordinates
$y_j^{(i)},
\,\,
i=1,\ldots,T-1,
\,\,
j=2,\ldots,N$. The coordinates $y_1^{(i)}$ and all primal coordinates are left
deterministic. By Part~5 of Lemma~\ref{lem:stoch-clipped-package}, every
randomized coordinate satisfies, for every feasible $\vct{z}$,
\begin{equation}\label{eq:stoch-randomized-gradient}
\left|
\nabla_{y_j^{(i)}}f^{\rm sg}(\vct{z})
\right|
\le
G_N,
\qquad
G_N:=2\ell_0\tau_N,
\qquad
j=2,\ldots,N.
\end{equation}

At an oracle call, if the next coordinate in the zero-chain ordering is
one of the randomized dual coordinates, we multiply only that gradient
component by $\xi/p_N$, where
$\xi\sim\mathrm{Bernoulli}(p_N)$. All remaining gradient components are
returned exactly. If the next coordinate is not randomized, the exact
gradient is returned.

Since all nonrandomized gradient components are returned exactly, the
difference $\widehat{\nabla}f^{\rm sg}(\vct{z};\xi)-\nabla f^{\rm sg}(\vct{z})$ is supported only on the randomized dual coordinate. Therefore,
the bound in~\eqref{eq:stoch-randomized-gradient} and the Bernoulli
calculation in Lemma~\ref{lem:bernoulli-zero-chain} give
\begin{equation}\label{eq:stoch-oracle-unbiased}
\mathbb E_\xi
\left[
\widehat{\nabla}f^{\rm sg}(\vct{z};\xi)
\right]
=
\nabla f^{\rm sg}(\vct{z}),
\end{equation}
and
\begin{equation}\label{eq:stoch-oracle-variance}
\mathbb E_\xi
\left[
\left\|
\widehat{\nabla}f^{\rm sg}(\vct{z};\xi)
-
\nabla f^{\rm sg}(\vct{z})
\right\|^2
\right]
\le
G_N^2\frac{1-p_N}{p_N}.
\end{equation}

For a prescribed variance bound $\sigma^2$, choose
\begin{equation}\label{eq:stoch-pN}
p_N
:=
\begin{cases}
1,
& \sigma=0,\\[1ex]
\min\{1,G_N^2/\sigma^2\},
& \sigma>0.
\end{cases}
\end{equation}
Then
\[
G_N^2\frac{1-p_N}{p_N}
\le
\sigma^2,
\]
so the resulting oracle is unbiased with mean-square error at most
$\sigma^2$.

We next consider the progress of a stochastic zero-respecting algorithm.
Each dual block contains $N-1$ randomized coordinates,
\[
y_2^{(i)},\ldots,y_N^{(i)},
\]
so the total number of randomized dual coordinates that must be crossed
before the terminal primal coordinate $u_T$ can be activated is
\begin{equation}\label{eq:stoch-randomized-length}
M
:=
(T-1)(N-1).
\end{equation}
The transitions through $u_i$, $a_i$, $y_1^{(i)}$, and $b_i$ are
deterministic. Treating these deterministic transitions as free can only
accelerate the algorithm. Hence, after suppressing them, progress is
measured through the ordered sequence
\[
y_2^{(1)},\ldots,y_N^{(1)},
y_2^{(2)},\ldots,y_N^{(2)},
\ldots,
y_2^{(T-1)},\ldots,y_N^{(T-1)}.
\]
By the zero-chain property in Lemma~\ref{lem:stoch-clipped-package}, these
coordinates cannot be skipped, while each next randomized coordinate is
revealed only after a successful Bernoulli trial with probability $p_N$.
Therefore the progress bound of
Lemma~\ref{lem:probability-zero-chain-progress} applies to this contracted
sequence.

Taking $\delta=1/4$ and replacing the chain length in
Lemma~\ref{lem:probability-zero-chain-progress} by $M$ gives
\begin{equation}\label{eq:stoch-progress-budget}
q
\le
\frac{M-\log 4}{2p_N}
\end{equation}
and consequently
\begin{equation}\label{eq:stoch-hidden-terminal}
\mathbb P(u_T=0)
\ge
\frac34.
\end{equation}
Since
\[
M=(T-1)(N-1)=\Theta(TN),
\]
contracting the deterministic transitions changes only universal
constants in the final oracle-complexity bound.

By Lemma~\ref{lem:stoch-clipped-package}, the clipped construction has the same primal value function, and hence the same proximal mapping, as the
deterministic construction. Therefore the stationarity lower bound in
Lemma~\ref{lem:gradient-lower} also applies to $f^{\rm sg}$. Choose
\begin{equation}\label{eq:stoch-scale}
s
:=
\frac{8\epsilon}{3C_\delta\ell_0}.
\end{equation}
Since $C_\ell\ell_0=L$, Part~3 of
Lemma~\ref{lem:gradient-lower} gives
\[
\left\|
\nabla(\varphi_{\rm sg})_{1/(2L)}(\vct{x})
\right\|
>
\frac{C_\delta}{2}\ell_0s
=
\frac43\epsilon
\]
whenever $u_T=0$. Under \eqref{eq:stoch-progress-budget}, this event has
probability at least $3/4$. Taking expectations therefore yields the
following lemma.
\begin{lemma}\label{lem:stoch-expected-obstruction}
Assume \eqref{eq:stoch-dual-feas}, the scale
\eqref{eq:stoch-scale}, and the progress budget
\eqref{eq:stoch-progress-budget}. Let $\vct{x}$ be any
zero-respecting primal output produced after $q$ oracle calls to the
stochastic oracle defined above, and let
\[
\vct{p}
=
p_{1/(2L)}(\vct{x})
\]
denote the unique proximal point of $\varphi_{\rm sg}$ at scale
$1/(2L)$. Then
\[
\mathbb E
\left\|
G_{1/(2L)}(\vct{x};\vct{p})
\right\|
>
\epsilon.
\]
\end{lemma}

\subsection{The Stochastic Lower Bound}

We now choose the chain lengths to satisfy the prescribed
initial-gap and dual-diameter bounds. With
\(\ell_0=L/C_{\ell}\) and \(s\) given by
\eqref{eq:stoch-scale}, these choices have the orders
\[
 T=\Theta\!\left(L\Del\,\eps^{-2}\right),
 \qquad
 N=\Theta\!\left(L\Dy\,\eps^{-1}\right).
\]
The number of randomized dual coordinates is therefore
\[
 M=(T-1)(N-1)=\Theta(TN).
\]
The probability \(p_N\) is chosen as in
\eqref{eq:stoch-pN} to satisfy the variance bound.
Combining the resulting progress bound with
Lemma~\ref{lem:stoch-expected-obstruction} gives the following
theorem.

\begin{theorem}[Stochastic zero-respecting lower bound]
\label{thm:stoch-zr}
There exist universal constants
\(c_{\rm d},c_{\rm n},c_1>0\) such that, for every
\(L,\Del,\Dy>0\), every \(\sigma\ge0\), and every \(\eps\)
satisfying
\begin{equation}\label{eq:stoch-smallness}
 0<\eps\le
 c_1\min\{\sqrt{L\Del},\,L\Dy\},
\end{equation}
there is an instance
\(f^{\rm sg}\in\mathcal F_{\rm NCC}(L,\Del,\Dy)\)
on finite-dimensional closed convex sets \(X_0,Y_0\), with
\(\operatorname{diam}(Y_0)=\Dy\), equipped with an unbiased
stochastic first-order oracle whose mean-square error is at most
\(\sigma^2\). For this instance and oracle, every zero-respecting
primal output \(\vct{x}\in X_0\) produced by an adaptive stochastic
zero-respecting algorithm after \(q\) oracle calls with
\begin{equation}\label{eq:stoch-query-lower}
 q<
 c_{\rm d}L^2\Dy\Del\,\eps^{-3}
 +
 c_{\rm n}L^3\Dy^2\Del\sigma^2\,\eps^{-6}
\end{equation}
satisfies
\begin{equation}\label{eq:stoch-main-obstruction}
 \mathbb E\!\left[
 \norm{\nabla(\varphi_{\rm sg})_{1/(2L)}(\vct{x})}
 \right]
 >
 \eps.
\end{equation}
Consequently, finding an expected \(\eps\)-stationary point
requires
\begin{equation}\label{eq:stoch-omega}
 \Omega\!\left(
 L^2\Dy\Del\,\eps^{-3}
 +
 L^3\Dy^2\Del\sigma^2\,\eps^{-6}
 \right)
\end{equation}
stochastic first-order oracle evaluations.
\end{theorem}

The dependence on the problem parameters follows from the gradient
bound \(G_N\) and the number \(M\) of randomized dual coordinates.
Since
\[
G_N=2\ell_0\tau_N,
\qquad
\tau_N=R\alpha s,
\qquad
\alpha^2N=1,
\]
we have, under \eqref{eq:stoch-scale},
\[
 G_N^2N
 =
 4R^2\ell_0^2s^2
 =
 \Theta(\eps^2).
\]
With the choice of \(p_N\) in \eqref{eq:stoch-pN}, including
\(p_N=1\) when \(\sigma=0\),
\[
\begin{aligned}
 \frac{M}{p_N}
 &=
 \Theta\!\left(
 TN + TN\,\frac{\sigma^2}{G_N^2}
 \right) \\
 &=
 \Theta\!\left(
 TN+TN^2\sigma^2\eps^{-2}
 \right) \\
 &=
 \Theta\!\left(
 L^2\Dy\Del\,\eps^{-3}
 +
 L^3\Dy^2\Del\sigma^2\,\eps^{-6}
 \right).
\end{aligned}
\]
For the parameter choices in the proof, \(T,N\ge4\), and hence
\(M\ge9\). Therefore
\[
q p_N\le \frac{M}{4}
\]
implies
\[
q\le\frac{M-\log 4}{2p_N},
\]
which is precisely the progress condition
\eqref{eq:stoch-progress-budget}. Lemma~\ref{lem:stoch-expected-obstruction}
then rules out expected \(\eps\)-stationarity. This yields the stated
oracle lower bound. The full proof, including the choices of integer
chain lengths and verification of the function-class conditions, is
given in Appendix~\ref{app:stochastic-theorem}.

The same construction gives a lower bound under the
primal--dual-gap parameterization.
Clipping preserves both the primal value function and
the function values at $\vct{y}=0$, so the initial
primal--dual gap still equals the initial primal value gap.
Applying Theorem~\ref{thm:stoch-zr} with
$\Del=\mathcal G_0$ therefore yields the following corollary.
Its proof is given in Appendix~\ref{app:stoch-pd-gap}.

\begin{corollary}\label{cor:g0-lower-bound}
There exists a universal constant $c>0$ such that, for every
$L,\mathcal G_0,\Dy>0$, every $\sigma\ge0$, and every $\eps$
satisfying
\[
0<\eps\le c\min\{\sqrt{L\mathcal G_0},L\Dy\},
\]
there is an instance
\[
f^{\rm sg}\in
\mathcal F_{\rm NCC}^{\rm pd}(L,\mathcal G_0,\Dy)
\]
equipped with a unbiased stochastic first-order oracle whose mean-square error is at most $\sigma^2$, for which every adaptive stochastic
zero-respecting algorithm requires
\[
\Omega\!\left(
L^2\Dy\mathcal G_0\,\eps^{-3}
+
L^3\Dy^2\mathcal G_0\sigma^2\,\eps^{-6}
\right)
\]
oracle calls to return an expected $\eps$-stationary output.
\end{corollary}

\section{Conclusion and Future Work}

We established first-order oracle lower bounds for smooth
nonconvex--concave minimax optimization within the class of
zero-respecting algorithms. Under the primal-value-gap and
primal--dual-gap parameterizations, the deterministic lower bounds are
$\Omega(L^2\Dy\Del\,\eps^{-3})$ and
$\Omega(L^2\Dy\mathcal G_0\,\eps^{-3})$, respectively.
In the stochastic setting, the corresponding lower bounds are
$\Omega(L^3\Dy^2\Del\sigma^2\,\eps^{-6})$ and
$\Omega(L^3\Dy^2\mathcal G_0\sigma^2\,\eps^{-6})$.
In their dependence on $L$, $\Dy$, and $\eps$, these bounds match
the deterministic upper bound of \cite{lin2020nearoptimal}
up to a logarithmic factor and the stochastic upper bound
of \cite{zhang2022sapdplus}.

Several questions remain open:

\paragraph{Beyond zero-respecting algorithms.}
Our lower bounds are proved for first-order zero-respecting methods. A natural next step is to determine whether the same rates hold for arbitrary, possibly randomized, first-order algorithms. In classical nonconvex optimization, zero-chain lower bounds can often be extended using random rotations or resisting-oracle arguments. Establishing an analogous reduction while preserving the primal--dual structure of the present construction is an important open problem.

\paragraph{Separate oracle complexities.}
Our lower bounds concern the joint first-order oracle, where each query
returns both $\nabla_x f$ and $\nabla_y f$. It remains open whether the
primal and dual gradient complexities differ when counted separately.
The structure of our hard instance suggests possible primal and dual
gradient complexities scaling as $\epsilon^{-2}$ and $\epsilon^{-3}$,
respectively, in the deterministic setting, and as $\epsilon^{-2}$ and
$\epsilon^{-6}$ in the bounded-variance stochastic setting. The best-known existing algorithms, however, require the same order of primal and dual gradient evaluations.  Resolving this
question may require new lower-bound constructions tailored to separate
primal and dual first-order oracles, new algorithms with improved primal or
dual gradient complexity, or both. Such results would clarify whether
existing joint-oracle methods are individually optimal in their use of
primal and dual gradients.

\paragraph{$f$-stationarity versus $\Phi$-stationarity.}
Recent algorithmic results have demonstrated that game stationarity and optimization stationarity may exhibit genuinely different complexity behavior \cite{li2025nonsmooth,li2026smoothing}. Our lower bounds concern stationarity of the value function
\(
\Phi(\vct{x})=\max_{\vct{y}\in \mathcal{Y}}f(\vct{x},\vct{y})
\). It remains open whether comparable lower bounds hold for stationarity notions defined directly in terms of the saddle function $f$, such as $\|\nabla f\|$. Since an $f$-stationary pair need not correspond to a nearly maximizing dual variable, the complexity of $f$-stationarity may differ substantially from that of $\Phi$-stationarity. Determining the tight complexity of the two notions, and whether a genuine separation exists between them, is an interesting open question.

\paragraph{NC--C Bilevel Optimization.}
An important direction for future research is to extend the present lower-bound framework to NC--C bilevel optimization, for which general first-order complexity lower bounds remain open. A natural question is whether the primal--dual information propagation mechanism developed here can be transferred to this setting and used to establish lower bounds for NC--C bilevel problems.

\section*{AI Disclosure}

The hard-instance construction was developed through iterative collaboration
with OpenAI's GPT-5.6 Sol Ultra. The authors provided the model with an
independently developed proof draft that already contained the central
dual-chain construction, the $\ell_2$ constraint on the auxiliary primal
variables, and the scaling argument underlying the $\epsilon^{-3}$
deterministic lower bound. The model was then used to help identify and
refine several composite functions in the primal component, which addressed
a key obstacle in the original construction. For the stochastic lower bound,
the authors independently introduced the stochastic dual zero-chain mechanism
used in the final proof. Building on these ingredients, the model assisted in
refining the construction and in developing several technical arguments,
including parts of the clipping mechanism. The authors subsequently simplified
and independently verified the resulting arguments and take full responsibility
for the correctness and presentation of the paper.
A Lean formalization, developed with Codex and available at \url{https://github.com/Wu-Qilong/Lower-Bounds-for-Nonconvex-Concave-Minimax-Optimization}, provides a formal verification.

\medskip

\bibliographystyle{plain}
\bibliography{ncc_minimax}
 
\newpage
 
\appendix

\section{Proofs for the Setup and Hard Instance}\label{app:auxiliary}

\subsection{Proof of Lemma~\ref{lemma:relay-properties}}\label{relay-properties}

\begin{proof}
The formulas in \eqref{eq:relay-maps} agree in value and first derivative at
$0$ and $1$, so both functions are $C^1$.  On $(0,1)$,
\[
 \nu'(t)=6t(1-t)=6r(t),\qquad
 \nu''(t)=6-12t,\qquad r'(t)=1-2t.
\]
Outside $(0,1)$, $\nu'$ is zero and $r'$ is constant with magnitude one.
This gives \eqref{eq:scalar-bounds} and
\eqref{eq:origin-identities}.

It remains to verify the two transition estimates.  If $t\le0$, then
$\nu(t)=0$; if $0<t\le2\kappa<1/4$, then
\[
 \frac{\nu(t)}{r(t)}
 =\frac{t(3-2t)}{1-t}\le2,
 \qquad
 r'(t)=1-2t\ge1-4\kappa.
\]
Similarly, if $t\ge1$, then $1-\nu(t)=0$; while for
$1-2\kappa\le t<1$,
\[
 \frac{1-\nu(t)}{r(t)}
 =\frac{(1-t)(1+2t)}{t}\le2.
\]
The claimed implications follow.
\end{proof}
\subsection{Proof of Lemma~\ref{lemma:relay-derivatives}}\label{app:relay-derivatives}
\begin{proof}
Differentiating \eqref{eq:q-rho},
\[
 Dq_i(\vct{u})[\vct{h}]
 =\nu'(u_i)(1-\nu(u_{i+1}))h_i
 -\nu(u_i)\nu'(u_{i+1})h_{i+1}.
\]
Every row and every column of $D\vct{q}(\vct{u})$ has at most two nonzero entries, each
with magnitude at most $3/2$.  Therefore
\[
 \norm{D\vct{q}(\vct{u})}_{\rm op}
 \le\sqrt{\norm{D\vct{q}(\vct{u})}_1\norm{D\vct{q}(\vct{u})}_\infty}\le3.
\]
Set
\[
 A_i(\vct{u}):=\nu'(u_i)(1-\nu(u_{i+1})),\qquad
 B_i(\vct{u}):=\nu(u_i)\nu'(u_{i+1}).
\]
Using \eqref{eq:scalar-bounds},
\[
 |A_i(\vct{u})-A_i(\vct{v})|
 \le6|u_i-v_i|+\frac94|u_{i+1}-v_{i+1}|,
\]
and
\[
 |B_i(\vct{u})-B_i(\vct{v})|
 \le\frac94|u_i-v_i|+6|u_{i+1}-v_{i+1}|.
\]
The same two-row/two-column sparsity argument gives
\[
 \norm{D\vct{q}(\vct{u})-D\vct{q}(\vct{v})}_{\rm op}\le17\norm{\vct{u}-\vct{v}}.
\]

Let $\vct{n}(\vct{z}):=\vct{z}/\sqrt{1+\norm{\vct{z}}^2}$, so that $\vct{\rho}=\vct{n}\circ \vct{q}$.  Direct
differentiation yields
\[
 D\vct{n}(\vct{z})=\frac{I}{\sqrt{1+\norm{\vct{z}}^2}}
 -\frac{\vct{z}\vct{z}^\top}{(1+\norm{\vct{z}}^2)^{3/2}},
\]
hence $\norm{D\vct{n}(\vct{z})}_{\rm op}\le1$ and
$\norm{D\vct{n}(\vct{z})-D\vct{n}(\vct{z}')}_{\rm op}\le6\norm{\vct{z}-\vct{z}'}$.  Since $\vct{q}$ is
$3$-Lipschitz,
\begin{align*}
 \norm{D\vct{\rho}(\vct{u})-D\vct{\rho}(\vct{v})}_{\rm op}
 &\le
 \norm{D\vct{n}(\vct{q}(\vct{u}))-D\vct{n}(\vct{q}(\vct{v}))}_{\rm op}\norm{D\vct{q}(\vct{u})}_{\rm op}\\
 &\quad+\norm{D\vct{n}(\vct{q}(\vct{v}))}_{\rm op}
       \norm{D\vct{q}(\vct{u})-D\vct{q}(\vct{v})}_{\rm op}\\
 &\le 6\cdot3\cdot3\norm{\vct{u}-\vct{v}}+17\norm{\vct{u}-\vct{v}}\\
 &=71\norm{\vct{u}-\vct{v}}.
\end{align*}
The bound $\norm{D\vct{\rho}(\vct{u})}_{\rm op}\le3$ follows from
$\norm{D\vct{n}}_{\rm op}\le1$ and $\norm{D\vct{q}}_{\rm op}\le3$.
\end{proof}

\subsection{Proof of Lemma~\ref{lem:gap}}\label{app:gap}
\begin{proof}
We first compute the inverse of $B_{\alpha,N}$. Define
\[
G_{ij}:=\alpha^{-2}+\min\{i,j\}-1,
\qquad i,j=1,\ldots,N.
\]
For each fixed $j$, the $j$th column of $G$ increases by one at each
index up to $j$ and remains constant after index $j$. Since the
interior rows of $B_{\alpha,N}$ act as a discrete second-difference
operator, we have
\begin{align*}
&(1+\alpha^2)G_{1j}-G_{2j}
=
\mathbf 1_{\{j=1\}}, \\
&-G_{i-1,j}+2G_{ij}-G_{i+1,j}
=
\mathbf 1_{\{i=j\}},
\qquad
2\le i\le N-1,\\
&-G_{N-1,j}+G_{Nj}
=
\mathbf 1_{\{j=N\}}.
\end{align*}
Thus $B_{\alpha,N}G=I$, and hence
\begin{equation}\label{eq:B_inv}
(B_{\alpha,N}^{-1})_{ij}
=
\alpha^{-2}+\min\{i,j\}-1.
\end{equation}

Now let
$\vct{c}(a,b):=\alpha a\vct{e}_1
-\frac{\alpha}{2}b\vct{e}_N$.
The first-order condition for maximizing $h_N(a,b;\cdot)$ is
\[
B_{\alpha,N}\vct{y}=\vct{c}(a,b).
\]
Using \eqref{eq:B_inv}, its unique solution is
\begin{equation}\label{eq:y_*}
y_k^*(a,b)
=
\frac{a}{\alpha}
-\frac{b}{2\alpha}
-\frac{\alpha b}{2}(k-1),
\qquad
k=1,\ldots,N.
\end{equation}
Since $B_{\alpha,N}$ is positive definite,
\[
h_N(a,b;\vct{y}^*)-h_N(a,b;\vct{y})
=
\frac{\ell_0}{2}
(\vct{y}-\vct{y}^*)^\top B_{\alpha,N}(\vct{y}-\vct{y}^*)
\ge 0,
\]
with equality only at $\vct{y}=\vct{y}^*$. Thus $\vct{y}^*$ is the unique
unconstrained maximizer.

From \eqref{eq:B_inv},
\[
(B_{\alpha,N}^{-1})_{11}
=
(B_{\alpha,N}^{-1})_{1N}
=
\alpha^{-2},
\qquad
(B_{\alpha,N}^{-1})_{NN}
=
\alpha^{-2}+N-1.
\]
Since $\alpha^2=1/N$,
\[
\vct{c}(a,b)^\top B_{\alpha,N}^{-1}\vct{c}(a,b)
=
a^2-ab+\frac{b^2}{4}\left(2-\frac1N\right).
\]
Substituting $\vct{y}^*$ into $h_N$, we obtain
\begin{equation}\label{eq:dual-value}
\max_{\vct{y}\in\mathbb R^N} h_N(a,b;\vct{y})
=
\frac{\ell_0}{2}
\left(a-\frac b2\right)^2.
\end{equation}

We next verify that the unconstrained maximizer lies in $Y_0$.
Rewrite \eqref{eq:y_*} as
\[
y_k^*(a,b)
=
\frac{a-c_k b}{\alpha},
\qquad
c_k:=
\frac12\left(1+\frac{k-1}{N}\right)\in[0,1].
\]
Therefore
\[
|y_k^*(a,b)|^2
\le
2N(a^2+b^2),
\]
and, summing over $k=1,\ldots,N$ and $i=1,\ldots,T-1$,
\[
\sum_{i=1}^{T-1}
\|\vct{y}^*(a_i,b_i)\|^2
\le
2N^2\bigl(\|\vct{a}\|^2+\|\vct{b}\|^2\bigr)
\le
2N^2R^2s^2.
\]
By assumption, the right-hand side is at most $(D_{\mathcal{Y}}/4)^2$.
Since $Y_0$ is the Euclidean ball of radius $D_{\mathcal{Y}}/2$, the concatenated
unconstrained maximizer lies in the interior of $Y_0$.
Moreover, each $h_N(a_i,b_i;\cdot)$ has $\vct{y}$-Hessian
$-\ell_0B_{\alpha,N}\prec0$, so $f(\vct{x},\cdot)$ is strictly concave.
This proves the first claim.

We now compute the value function. Since the unconstrained maximizer of
the dual variables lies in $Y_0$, \eqref{eq:dual-value} gives
\begin{equation}\label{eq:phi_func}
\begin{aligned}
\Phi(\vct{u},\vct{a},\vct{b})
&=
\ell_0s^2\Psi_0(\tilde{\vct{u}},\tilde{\vct{a}},\tilde{\vct{b}})
+
\frac{\ell_0}{2}
\sum_{i=1}^{T-1}
\left(a_i-\frac{b_i}{2}\right)^2 \\
&=
\ell_0s^2
\left[
\Psi_0(\tilde{\vct{u}},\tilde{\vct{a}},\tilde{\vct{b}})
+
\frac12
\left\|
\tilde{\vct{a}}-\frac12\tilde{\vct{b}}
\right\|^2
\right] \\
&=
\ell_0s^2\Psi(\tilde{\vct{u}},\tilde{\vct{a}},\tilde{\vct{b}}).
\end{aligned}
\end{equation}
This proves the second claim.

It remains to compute the initial gap. Every squared term in $\Psi$
is nonnegative. For $j\ge2$,
\[
-c_\eta\nu(t)+\frac{c_\eta}{2}r(t)^2
\ge -c_\eta,
\]
while for the first coordinate,
\[
-c_\eta t+\frac{c_\eta}{2}r(t)^2
\ge -\frac32c_\eta.
\]
Hence
\[
\Psi(\tilde{\vct{u}},\tilde{\vct{a}},\tilde{\vct{b}})
\ge
-c_\eta\left(T+\frac12\right).
\]
Equality is attained at
\[
\tilde{\vct{u}}^*=(2,1,\ldots,1),
\qquad
\tilde{\vct{a}}=\tilde{\vct{b}}=\vct{0},
\]
for which $\vct{\rho}(\tilde{\vct{u}}^*)=\bar{\vct{r}}(\tilde{\vct{u}}^*)=0$.
Since $\Psi(\vct{0},\vct{0},\vct{0})=0$, \eqref{eq:phi_func} yields
\[
\Phi(\vct{0})-\inf_{\vct{x}\in X_0}\Phi(\vct{x})
=
c_\eta\left(T+\frac12\right)\ell_0s^2.
\]
This proves the third claim.
\end{proof}

\subsection{Proof of Lemma~\ref{lem:smooth}}\label{app:smooth}
\begin{proof}
    Fix
\[
\tilde{\vct{z}}
=(\tilde{\vct{u}},\tilde{\vct{a}},\tilde{\vct{b}})
\in C_0,
\qquad
\tilde{\vct{z}}'
=(\tilde{\vct{u}}',\tilde{\vct{a}}',\tilde{\vct{b}}')
\in C_0.
\]
In particular,
\[
\|\tilde{\vct{a}}\|,\|\tilde{\vct{a}}'\|,
\|\tilde{\vct{b}}\|,\|\tilde{\vct{b}}'\|
\le R.
\]

We first bound the Lipschitz constant of the gradient of the
normalized primal component $\Psi_0$. We define
\[
\begin{aligned}
&F_a(\tilde{\vct{u}},\tilde{\vct{a}})
=
\frac12
\|\tilde{\vct{a}}-\vct{\rho}(\tilde{\vct{u}})\|^2,
\\
&F_b(\tilde{\vct{u}},\tilde{\vct{b}})
=
\frac12
\|\tilde{\vct{b}}-\bar{\vct{r}}(\tilde{\vct{u}})\|^2,
\\
&Q(\tilde{\vct{a}},\tilde{\vct{b}})
=
\frac12
\bigl(
\|\tilde{\vct{a}}\|^2+\|\tilde{\vct{b}}\|^2
\bigr),
\end{aligned}
\]
and thus $\Psi_0=H+F_a+F_b+Q$.
By Lemma~\ref{lemma:relay-properties}, $\nu'$ is $6$-Lipschitz
and $t\mapsto r(t)r'(t)$ is $2$-Lipschitz. Hence
\begin{equation}\label{eq:lip_H}
\|\nabla H(\tilde{\vct{u}})
-\nabla H(\tilde{\vct{u}}')\|
\le
8c_\eta
\|\tilde{\vct{u}}-\tilde{\vct{u}}'\|.
\end{equation}

For $F_a$, we have
\[
\nabla_{\tilde{\vct{a}}}F_a
=
\tilde{\vct{a}}-\vct{\rho}(\tilde{\vct{u}}),
\qquad
\nabla_{\tilde{\vct{u}}}F_a
=
-D\vct{\rho}(\tilde{\vct{u}})^\top
\bigl(
\tilde{\vct{a}}-\vct{\rho}(\tilde{\vct{u}})
\bigr).
\]
Then, we define
$d_u:=\|\tilde{\vct{u}}-\tilde{\vct{u}}'\|$,
$d_a:=\|\tilde{\vct{a}}-\tilde{\vct{a}}'\|$, and
$d_b:=\|\tilde{\vct{b}}-\tilde{\vct{b}}'\|$.
Using Lemma~\ref{lemma:relay-derivatives},
$\|\vct{\rho}(\tilde{\vct{u}})\|\le1$, and
$\|\tilde{\vct{a}}\|,\|\tilde{\vct{a}}'\|\le R$,
we obtain
\[
\left\|
\nabla_{\tilde{\vct{a}}}
F_a(\tilde{\vct{u}},\tilde{\vct{a}})
-
\nabla_{\tilde{\vct{a}}}
F_a(\tilde{\vct{u}}',\tilde{\vct{a}}')
\right\|
\le
d_a+3d_u,
\]
and
\[
\left\|
\nabla_{\tilde{\vct{u}}}
F_a(\tilde{\vct{u}},\tilde{\vct{a}})
-
\nabla_{\tilde{\vct{u}}}
F_a(\tilde{\vct{u}}',\tilde{\vct{a}}')
\right\|
\le
3d_a+\bigl(9+71(R+1)\bigr)d_u.
\]
Thus, since $R$ is a constant, there exists a constant
$C_a=\sqrt{(80+71R)^2+19}$ such that
\begin{equation}\label{eq:lip_a}
\|\nabla F_a(\tilde{\vct{u}},\tilde{\vct{a}})
-
\nabla F_a(\tilde{\vct{u}}',\tilde{\vct{a}}')\|
\le
C_a\sqrt{d_u^2+d_a^2}.
\end{equation}

For $F_b$, it suffices to consider
$\vartheta(t,b):=\frac12(b-r(t))^2$.
Wherever the Hessian exists,
\[
\nabla^2\vartheta(t,b)
=
\begin{pmatrix}
r'(t)^2-(b-r(t))r''(t) & -r'(t)\\
-r'(t) & 1
\end{pmatrix}.
\]
Although $r$ is unbounded on $\R$, we have
$r''(t)=0$ for $t\notin[0,1]$, while
$|r(t)|\le1/4$ for $t\in[0,1]$.
Together with $|r'(t)|\le1$, $|r''(t)|\le2$ almost
everywhere, and $|b|\le R$, this gives
\[
\left|r'(t)^2-(b-r(t))r''(t)\right|
\le 2R+\frac32
\]
wherever the Hessian exists. Consequently,
\[
\|\nabla^2\vartheta(t,b)\|_{\rm op}
\le C_b,
\qquad
C_b:=\sqrt{\left(2R+\frac32\right)^2+3}.
\]
Since $\nabla\vartheta$ is continuous across $t=0,1$,
integrating along line segments in $\R\times[-R,R]$
shows that $\nabla\vartheta$ is $C_b$-Lipschitz on
this set. Applying this bound to the disjoint
coordinate pairs in $F_b$ yields
\begin{equation}\label{eq:lip_b}
\|\nabla F_b(\tilde{\vct{u}},\tilde{\vct{b}})
-
\nabla F_b(\tilde{\vct{u}}',\tilde{\vct{b}}')\|
\le
C_b\sqrt{d_u^2+d_b^2}.
\end{equation}

Finally, $\nabla Q$ is $1$-Lipschitz. Therefore there exists
a constant $C_\Psi=8c_\eta+C_a+C_b+1>0$ such that
\begin{equation}\label{eq:lip_psi}
\|\nabla\Psi_0(\tilde{\vct{z}})
-
\nabla\Psi_0(\tilde{\vct{z}}')\|
\le
C_\Psi
\|\tilde{\vct{z}}-\tilde{\vct{z}}'\|.
\end{equation}

For $\vct{x}=(\vct{u},\vct{a},\vct{b})\in X_0$ and
$\vct{x}'=(\vct{u}',\vct{a}',\vct{b}')\in X_0$, the first term
of the hard instance depends only on the primal variables
and is given by $\ell_0s^2\Psi_0(\vct{x}/s)$.
By the chain rule and~\eqref{eq:lip_psi}, we have
\begin{equation}\label{eq:lip_psi_z}
\left\|
\nabla\left[\ell_0s^2\Psi_0(\vct{x}/s)\right]
-
\nabla\left[\ell_0s^2\Psi_0(\vct{x}'/s)\right]
\right\|
\le
C_\Psi\ell_0\|\vct{x}-\vct{x}'\|.
\end{equation}

It remains to bound the dual component terms.
For one block $h_N$, in the coordinate order $(a,\vct{y},b)$,
the Hessian divided by $\ell_0$ is
\[
\begin{pmatrix}
0
& \alpha\vct{e}_1^\top
& 0
\\
\alpha\vct{e}_1
& -B_{\alpha,N}
& -(\alpha/2)\vct{e}_N
\\
0
& -(\alpha/2)\vct{e}_N^\top
& -\alpha^2(N-1)/4
\end{pmatrix}.
\]
Since $N\ge2$ and $\alpha=N^{-1/2}\le1$, every absolute
row sum of this symmetric matrix is at most $4$.
Hence its operator norm is at most $4$, and therefore
\begin{equation}\label{eq:lip_h}
\|\nabla h_N(a,b;\vct{y})
-
\nabla h_N(a',b';\vct{y}')\|
\le
4\ell_0
\|(a,b,\vct{y})-(a',b',\vct{y}')\|.
\end{equation}
The different path blocks act on disjoint coordinates,
so the same bound holds for their sum.

Combining \eqref{eq:lip_psi_z} and \eqref{eq:lip_h},
we obtain
\[
\|\nabla f(\vct{z})-\nabla f(\vct{z}')\|
\le
(C_\Psi+4)\ell_0\|\vct{z}-\vct{z}'\|.
\]
Fix a universal constant $C_\ell$ satisfying
\[
C_\ell\ge \max\{C_\Psi+6,C_\delta+1\},
\]
where $C_\delta$ is the universal constant in
Lemma~\ref{lem:gradient-lower}.
This choice proves the claim and will also be used for
the stochastic construction.
\end{proof}

\subsection{Proof of Lemma~\ref{lem:gradient-lower}}\label{app:zero-chain}

\begin{proof}
\emph{Part 1: zero-chain property.}
Fix $\vct{z}=(\vct{x},\vct{y})\in X_0\times Y_0$,
where $\vct{x}=(\vct{u},\vct{a},\vct{b})$, and set
$m:=\prog^{\vct{\pi}}_0(\vct{z})$.
The desired inequality is immediate if $m\ge K-1$.
Otherwise, all coordinates at positions greater than $m$
are zero. We show that
\[
[\nabla f(\vct{z})]_{\pi_k}=0,
\qquad k=m+2,\ldots,K.
\]
Since $s>0$, normalization preserves the support of the
primal variables.

We first record two properties of the normalized coupling.
The definition of $\vct{q}$ and $\nu(0)=0$ imply
\[
\tilde u_i=0
\quad\Longrightarrow\quad
q_i(\tilde{\vct{u}})
=\rho_i(\tilde{\vct{u}})=0.
\]
Moreover, every potentially nonzero entry in the $j$th
column of $D\vct{q}(\tilde{\vct{u}})$ contains the factor
$\nu'(\tilde u_j)$. Thus, $\nu'(0)=0$ and the chain rule give
\[
\tilde u_j=0
\quad\Longrightarrow\quad
D\vct{q}(\tilde{\vct{u}})\vct{e}_j=0,
\qquad
D\vct{\rho}(\tilde{\vct{u}})\vct{e}_j=0.
\]

Fix $k\ge m+2$. We distinguish the four possible types
of $\pi_k$ in the ordering $\vct{\pi}$.

\emph{Case 1: $\pi_k=u_j$.}
Then $j\ge2$. Since $b_{j-1}$ immediately precedes $u_j$,
both coordinates are zero.
At $\tilde u_j=0$, the derivative of $H$ with respect to
$\tilde u_j$ vanishes, as does the contribution involving
$D\vct{\rho}(\tilde{\vct{u}})\vct{e}_j$.
The dual blocks do not depend on $\vct{u}$. Hence
\[
\partial_{u_j}f(\vct{z})
=
-\ell_0s
\bigl(\tilde b_{j-1}-r(0)\bigr)r'(0)
=
-\ell_0b_{j-1}
=0.
\]

\emph{Case 2: $\pi_k=a_i$.}
The adjacent coordinates are $u_i$ and $y_1^{(i)}$,
so $u_i=a_i=y_1^{(i)}=0$.
In particular, $\rho_i(\tilde{\vct{u}})=0$, and
\[
\partial_{a_i}f(\vct{z})
=
\ell_0s
\bigl(2\tilde a_i-\rho_i(\tilde{\vct{u}})\bigr)
+\ell_0\alpha y_1^{(i)}
=0.
\]

\emph{Case 3: $\pi_k=y_j^{(i)}$.}
Differentiating the dual block gives
\[
\ell_0^{-1}\partial_{y_j^{(i)}}f(\vct{z})
=
\begin{cases}
\alpha a_i-(1+\alpha^2)y_1^{(i)}+y_2^{(i)},
& j=1,\\[1mm]
y_{j-1}^{(i)}-2y_j^{(i)}+y_{j+1}^{(i)},
& 2\le j\le N-1,\\[1mm]
y_{N-1}^{(i)}-y_N^{(i)}-\dfrac{\alpha}{2}b_i,
& j=N.
\end{cases}
\]
Each expression involves only coordinates at positions
$k-1$, $k$, and $k+1$, all greater than $m$.
Thus $\partial_{y_j^{(i)}}f(\vct{z})=0$.

\emph{Case 4: $\pi_k=b_i$.}
The adjacent coordinates are $y_N^{(i)}$ and $u_{i+1}$,
so $y_N^{(i)}=b_i=u_{i+1}=0$. Therefore
\[
\partial_{b_i}f(\vct{z})
=
\ell_0s
\bigl(2\tilde b_i-r(\tilde u_{i+1})\bigr)
-\frac{\ell_0\alpha}{2}y_N^{(i)}
-\frac{\ell_0\alpha^2(N-1)}4b_i
=0.
\]

Consequently,
\[
\prog^{\vct{\pi}}_0\bigl(\nabla f(\vct{z})\bigr)
\le
\min\{\prog^{\vct{\pi}}_0(\vct{z})+1,K\},
\qquad
\vct{z}\in X_0\times Y_0.
\]
This proves part~1.

\emph{Part 2: normalized stationarity obstruction.}
For the second claim, fix
$(\tilde{\vct{u}},\tilde{\vct{a}},\tilde{\vct{b}})\in C_0$
with $\tilde u_T\le1/4$.
We argue by contradiction. Suppose that
$\Rcal_{C_0}(\tilde{\vct{u}},\tilde{\vct{a}},\tilde{\vct{b}})
<C_\delta$.
Then there exists
$\vct{n}=(\vct{n}_u,\vct{n}_{ab})
\in N_{C_0}(\tilde{\vct{u}},\tilde{\vct{a}},\tilde{\vct{b}})$
such that
\begin{equation}\label{eq:small_residual}
\|\nabla\Psi(\tilde{\vct{u}},\tilde{\vct{a}},\tilde{\vct{b}})
+\vct{n}\|
<
C_\delta.
\end{equation}
Since
\[
C_0
=
\mathbb R^T\times
\left\{
(\tilde{\vct{a}},\tilde{\vct{b}}):
\|\tilde{\vct{a}}\|^2+\|\tilde{\vct{b}}\|^2\le R^2
\right\},
\]
we have $\vct{n}_u=0$. Hence
\begin{equation}\label{eq:small_u_grad}
\|\nabla_{\tilde{\vct{u}}}
\Psi(\tilde{\vct{u}},\tilde{\vct{a}},\tilde{\vct{b}})\|
<
C_\delta.
\end{equation}

We first derive the implications of \eqref{eq:small_u_grad}
for the coordinates of $\tilde{\vct{u}}$.
Partition the tail coordinates as
\[
I_{\rm out}
:=
\{j\in\{2,\ldots,T\}:\tilde u_j\notin(0,1)\},
\qquad
I_{\rm mid}
:=
\{j\in\{2,\ldots,T\}:0<\tilde u_j<1\}.
\]

For $j\in I_{\rm out}$,
$\nu'(\tilde u_j)=0$ and the $j$th column of
$D\vct{\rho}(\tilde{\vct{u}})$ vanishes.
Moreover, $|r'(\tilde u_j)|=1$, and therefore
\[
(\nabla_{\tilde{\vct{u}}}\Psi)_j
=
r'(\tilde u_j)
\bigl((c_\eta+1)r(\tilde u_j)-\tilde b_{j-1}\bigr).
\]
Using \eqref{eq:small_u_grad} and
$\|\tilde{\vct{b}}\|\le R$ gives
\begin{equation}\label{eq:r_Iout_bound}
(c_\eta+1)
\left(
\sum_{j\in I_{\rm out}}r(\tilde u_j)^2
\right)^{1/2}
\le
C_\delta+R.
\end{equation}

For $j\in I_{\rm mid}$, let $t=\tilde u_j$.
Since $\nu'(t)=6r(t)$ and $r'(t)=1-2t$,
\[
(\nabla_{\tilde{\vct{u}}}\Psi)_j
=
-\gamma_j r(t)
-\tilde b_{j-1}r'(t)
-\left\langle
\tilde{\vct{a}}-\vct{\rho}(\tilde{\vct{u}}),
D\vct{\rho}(\tilde{\vct{u}})\vct{e}_j
\right\rangle,
\]
where
\[
\gamma_j
:=
c_\eta(5+2t)-r'(t)
\ge
5c_\eta-1.
\]
Using \eqref{eq:small_u_grad},
Lemma~\ref{lemma:relay-derivatives},
$\|\tilde{\vct{a}}\|,\|\tilde{\vct{b}}\|\le R$, and
$\|\vct{\rho}(\tilde{\vct{u}})\|\le1$, we obtain
\begin{equation}\label{eq:r_Imid_bound}
(5c_\eta-1)
\left(
\sum_{j\in I_{\rm mid}}r(\tilde u_j)^2
\right)^{1/2}
\le
4R+3+C_\delta.
\end{equation}

Since $I_{\rm out}$ and $I_{\rm mid}$ partition
$\{2,\ldots,T\}$, \eqref{eq:r_Iout_bound} and
\eqref{eq:r_Imid_bound} imply
\begin{align}
\|\bar{\vct{r}}(\tilde{\vct{u}})\|
&=
\left(
\sum_{j\in I_{\rm out}}r(\tilde u_j)^2
+
\sum_{j\in I_{\rm mid}}r(\tilde u_j)^2
\right)^{1/2}
\nonumber\\
&\le
\left(
\sum_{j\in I_{\rm out}}r(\tilde u_j)^2
\right)^{1/2}
+
\left(
\sum_{j\in I_{\rm mid}}r(\tilde u_j)^2
\right)^{1/2}
\nonumber\\
&\le
\frac{R+C_\delta}{c_\eta+1}
+
\frac{4R+3+C_\delta}{5c_\eta-1}.
\label{eq:rbar_prebound}
\end{align}
Define
\begin{equation}\label{eq:kappa}
\kappa
:=
\frac{R+C_\delta}{c_\eta+1}
+
\frac{4R+3+C_\delta}{5c_\eta-1}.
\end{equation}
Then
\begin{equation}\label{eq:rbar_bound}
\|\bar{\vct{r}}(\tilde{\vct{u}})\|\le\kappa.
\end{equation}

We choose the universal construction constants
$R\ge1$, $C_\delta>0$, and $c_\eta>0$ so that
\begin{equation}\label{eq:constant_conditions}
0<\kappa<\frac18,
\qquad
\frac34c_\eta-3(R+1)>C_\delta,
\end{equation}
and
\begin{equation}\label{eq:direction_condition}
12\kappa
\left(
\frac{17}{26}+\frac{\kappa}{13}+C_\delta
\right)
+\frac7{13}\kappa
+2C_\delta
<
\frac{1-4\kappa}{13}
\frac{(1-4\kappa)^2}
     {1+(1-4\kappa)^2+4\kappa^2}.
\end{equation}
Such universal choices exist: for fixed $R\ge1$ and
sufficiently small $C_\delta$, taking $c_\eta$ sufficiently
large makes $\kappa$ arbitrarily small, while the right-hand
side of \eqref{eq:direction_condition} converges to $1/26$.

We now show that every tail coordinate is either close to
zero or close to one. In fact,
\begin{equation}\label{eq:low_high_dichotomy}
|\tilde u_j|\le2\kappa
\quad\text{or}\quad
\tilde u_j\ge1-2\kappa,
\qquad
j=2,\ldots,T.
\end{equation}
Indeed, \eqref{eq:rbar_bound} gives
$|r(\tilde u_j)|\le\kappa$.
If $\tilde u_j\le0$, then
$|\tilde u_j|=|r(\tilde u_j)|\le\kappa$, while
$\tilde u_j\ge1$ satisfies the second alternative immediately.
If $0<\tilde u_j<1$, then
\[
\tilde u_j(1-\tilde u_j)\le\kappa.
\]
If both $\tilde u_j>2\kappa$ and
$\tilde u_j<1-2\kappa$, one of $\tilde u_j$ and
$1-\tilde u_j$ is at least $1/2$ and the other is larger
than $2\kappa$, giving
$\tilde u_j(1-\tilde u_j)>\kappa$, a contradiction.

We next show that the first coordinate satisfies the second
alternative:
\begin{equation}\label{eq:u1_high}
\tilde u_1\ge1-2\kappa.
\end{equation}
Otherwise $\tilde u_1<1-2\kappa<1$.
Since $r(t)r'(t)\le1/4$ for $t\le1$,
Lemma~\ref{lemma:relay-derivatives} gives
\[
\begin{aligned}
\partial_{\tilde u_1}\Psi
&=
-c_\eta
+c_\eta r(\tilde u_1)r'(\tilde u_1)
-\left\langle
\tilde{\vct{a}}-\vct{\rho}(\tilde{\vct{u}}),
D\vct{\rho}(\tilde{\vct{u}})\vct{e}_1
\right\rangle
\\
&\le
-\frac34c_\eta+3(R+1) \\
&<
-C_\delta,
\end{aligned}
\]
contradicting \eqref{eq:small_u_grad}.

We refer to the two alternatives in
\eqref{eq:low_high_dichotomy} as low and high, respectively.
Since $\tilde u_1$ is high and
$\tilde u_T\le1/4<1-2\kappa$, the terminal coordinate is low.
Thus there must be a high-to-low transition. Define
\begin{equation}\label{eq:J}
J
:=
\left\{
i\in\{1,\ldots,T-1\}:
\tilde u_i\ge1-2\kappa,\;
|\tilde u_{i+1}|\le2\kappa
\right\}.
\end{equation}
Then $J\neq\varnothing$. Let
\[
J^c:=\{1,\ldots,T-1\}\setminus J,
\]
and for any $I\subseteq\{1,\ldots,T-1\}$, let
$\vct{q}_I(\tilde{\vct{u}})$ and
$\vct{\rho}_I(\tilde{\vct{u}})$ denote the corresponding
subvectors.

We next lower bound the coupling across these transitions.
For $i\in J$, the definition of $\nu$ gives
\[
\nu(\tilde u_i)\ge1-2\kappa,
\qquad
1-\nu(\tilde u_{i+1})\ge1-2\kappa,
\]
and hence
\begin{equation}\label{eq:qJ}
q_i(\tilde{\vct{u}})
\ge
(1-2\kappa)^2
\ge
1-4\kappa.
\end{equation}
For $i\notin J$, either $\tilde u_i$ is low or
$\tilde u_{i+1}$ is high.
Lemma~\ref{lemma:relay-properties} then gives, respectively,
\[
q_i(\tilde{\vct{u}})\le2|r(\tilde u_i)|
\qquad\text{or}\qquad
q_i(\tilde{\vct{u}})\le2|r(\tilde u_{i+1})|.
\]
Assigning $i$ to the corresponding tail coordinate gives
an injective assignment, since no coordinate can be both
low and high. Therefore
\begin{equation}\label{eq:qJc}
\|\vct{q}_{J^c}(\tilde{\vct{u}})\|
\le
2\|\bar{\vct{r}}(\tilde{\vct{u}})\|
\le
2\kappa.
\end{equation}
Since $J\neq\varnothing$, \eqref{eq:qJ}--\eqref{eq:qJc}
imply
\begin{equation}\label{eq:rhoJ}
\|\vct{\rho}_J(\tilde{\vct{u}})\|^2
\ge
\frac{(1-4\kappa)^2}
     {1+(1-4\kappa)^2+4\kappa^2}.
\end{equation}

The residual bound also controls the auxiliary variables.
For fixed $\tilde{\vct{u}}$, the partial gradients with
respect to $(\tilde{\vct{a}},\tilde{\vct{b}})$ are
\begin{equation}\label{eq:ab_grad}
\begin{aligned}
\vct{g}_a
&=
3\tilde{\vct{a}}
-\frac12\tilde{\vct{b}}
-\vct{\rho}(\tilde{\vct{u}}),
\\
\vct{g}_b
&=
-\frac12\tilde{\vct{a}}
+\frac94\tilde{\vct{b}}
-\bar{\vct{r}}(\tilde{\vct{u}}).
\end{aligned}
\end{equation}
Their unique common zero is
\begin{equation}\label{eq:ab_star}
\begin{aligned}
\tilde{\vct{a}}^*
&=
\frac9{26}\vct{\rho}(\tilde{\vct{u}})
+\frac1{13}\bar{\vct{r}}(\tilde{\vct{u}}),
\\
\tilde{\vct{b}}^*
&=
\frac1{13}\vct{\rho}(\tilde{\vct{u}})
+\frac6{13}\bar{\vct{r}}(\tilde{\vct{u}}).
\end{aligned}
\end{equation}
For
$\vct{d}_a=\tilde{\vct{a}}-\tilde{\vct{a}}'$ and
$\vct{d}_b=\tilde{\vct{b}}-\tilde{\vct{b}}'$,
direct calculation gives
\[
\begin{aligned}
\left\langle
\vct{g}_a(\tilde{\vct{a}},\tilde{\vct{b}})
-\vct{g}_a(\tilde{\vct{a}}',\tilde{\vct{b}}'),
\vct{d}_a
\right\rangle+
\left\langle
\vct{g}_b(\tilde{\vct{a}},\tilde{\vct{b}})
-\vct{g}_b(\tilde{\vct{a}}',\tilde{\vct{b}}'),
\vct{d}_b
\right\rangle=
2\bigl(\|\vct{d}_a\|^2+\|\vct{d}_b\|^2\bigr)
+
\left\|\vct{d}_a-\frac12\vct{d}_b\right\|^2.
\end{aligned}
\]
Thus the $(\tilde{\vct{a}},\tilde{\vct{b}})$-gradient is
$2$-strongly monotone.

By \eqref{eq:rbar_bound} and
$\|\vct{\rho}(\tilde{\vct{u}})\|\le1$,
\[
\|\tilde{\vct{a}}^*\|
\le
\frac9{26}+\frac{\kappa}{13},
\qquad
\|\tilde{\vct{b}}^*\|
\le
\frac1{13}+\frac{6\kappa}{13}.
\]
Since $\kappa<1/8$ and $R\ge1$, these bounds imply
\[
\|\tilde{\vct{a}}^*\|^2
+
\|\tilde{\vct{b}}^*\|^2
<
R^2,
\]
so $(\tilde{\vct{a}}^*,\tilde{\vct{b}}^*)$ lies in the
interior of the radius-$R$ ball.

Define
\[
\vct{d}
:=
\bigl(
\tilde{\vct{a}}-\tilde{\vct{a}}^*,
\tilde{\vct{b}}-\tilde{\vct{b}}^*
\bigr).
\]
Since $(\tilde{\vct{a}}^*,\tilde{\vct{b}}^*)$ is feasible,
the normal-cone condition gives
\[
\langle\vct{n}_{ab},\vct{d}\rangle\ge0.
\]
Let $\vct{e}:=(\vct{g}_a,\vct{g}_b)+\vct{n}_{ab}$.
By \eqref{eq:small_residual}, $\|\vct{e}\|<C_\delta$.
Strong monotonicity therefore yields
\[
2\|\vct{d}\|^2
\le
\langle(\vct{g}_a,\vct{g}_b),\vct{d}\rangle
=
\langle\vct{e},\vct{d}\rangle
-
\langle\vct{n}_{ab},\vct{d}\rangle
\le
C_\delta\|\vct{d}\|,
\]
and hence
\begin{equation}\label{eq:ab_close}
\|\vct{d}\|\le\frac{C_\delta}{2}.
\end{equation}
Let
$\vct{e}_a:=\tilde{\vct{a}}-\tilde{\vct{a}}^*$, and
$\vct{e}_b:=\tilde{\vct{b}}-\tilde{\vct{b}}^*$.
Then \eqref{eq:ab_star} gives
\begin{equation}\label{eq:ab_expand}
\begin{aligned}
\tilde{\vct{a}}-\vct{\rho}(\tilde{\vct{u}})
&=
-\frac{17}{26}\vct{\rho}(\tilde{\vct{u}})
+\frac1{13}\bar{\vct{r}}(\tilde{\vct{u}})
+\vct{e}_a,
\\
\tilde{\vct{b}}-\bar{\vct{r}}(\tilde{\vct{u}})
&=
\frac1{13}\vct{\rho}(\tilde{\vct{u}})
-\frac7{13}\bar{\vct{r}}(\tilde{\vct{u}})
+\vct{e}_b,
\end{aligned}
\end{equation}
where $\|\vct{e}_a\|,\|\vct{e}_b\|\le C_\delta$.

We now use the high-to-low transitions to construct a
direction that contradicts \eqref{eq:small_u_grad}.
Define $\vct{v}\in\mathbb R^T$ by
\begin{equation}\label{eq:v}
v_1:=0,
\qquad
v_{i+1}:=
\begin{cases}
\rho_i(\tilde{\vct{u}}), & i\in J,\\
0, & i\notin J,
\end{cases}
\qquad i=1,\ldots,T-1.
\end{equation}
By \eqref{eq:rhoJ},
\begin{equation}\label{eq:v_norm}
\|\vct{v}\|^2
=
\|\vct{\rho}_J(\tilde{\vct{u}})\|^2
\ge
\frac{(1-4\kappa)^2}
     {1+(1-4\kappa)^2+4\kappa^2},
\qquad
\|\vct{v}\|\le1.
\end{equation}

The indices in $J$ are pairwise nonadjacent.
Moreover, for each $i\in J$, the coordinate
$\tilde u_{i+1}$ is low, so the $(i+1)$st column of
$D\vct{q}(\tilde{\vct{u}})$ has at most two nonzero
entries, each with magnitude at most $6\kappa$.
The supports of these columns are disjoint for distinct
$i\in J$. Since $\vct{v}$ is supported on the coordinates
$i+1$ with $i\in J$,
\[
D\vct{q}(\tilde{\vct{u}})\vct{v}
=
\sum_{i\in J}
v_{i+1}D\vct{q}(\tilde{\vct{u}})\vct{e}_{i+1}.
\]
Hence, by the disjointness of the supports,
\[
\begin{aligned}
\|D\vct{q}(\tilde{\vct{u}})\vct{v}\|^2
&=
\sum_{i\in J}
v_{i+1}^2
\|D\vct{q}(\tilde{\vct{u}})\vct{e}_{i+1}\|^2
\\
&\le
2(6\kappa)^2
\sum_{i\in J}v_{i+1}^2
=
2(6\kappa)^2\|\vct{v}\|^2.
\end{aligned}
\]
Therefore,
$\|D\vct{q}(\tilde{\vct{u}})\vct{v}\|
\le6\sqrt2\,\kappa\|\vct{v}\|$.
Since
\[
\vct{\rho}(\tilde{\vct{u}})
=
\frac{\vct{q}(\tilde{\vct{u}})}
{\sqrt{1+\|\vct{q}(\tilde{\vct{u}})\|^2}}
\]
and the Jacobian of the normalization map
$\vct{z}\mapsto\vct{z}/\sqrt{1+\|\vct{z}\|^2}$
has operator norm at most one,
\begin{equation}\label{eq:Drho_v}
\|D\vct{\rho}(\tilde{\vct{u}})\vct{v}\|
\le
6\sqrt2\,\kappa\|\vct{v}\|
\le
12\kappa\|\vct{v}\|.
\end{equation}

We finally evaluate the directional derivative along
$\vct{v}$. For every $j$ in the support of $\vct{v}$, we have
$j\ge2$, $\tilde u_j\le2\kappa<1$, and $v_j\ge0$.
Moreover,
\[
\partial_{\tilde u_j}H(\tilde{\vct{u}})
=
\begin{cases}
c_\eta\tilde u_j,
& \tilde u_j\le0,\\
-c_\eta(5+2\tilde u_j)r(\tilde u_j),
& 0<\tilde u_j\le2\kappa.
\end{cases}
\]
Both expressions are nonpositive. Therefore
$DH(\tilde{\vct{u}})[\vct{v}]\le0$. Using \eqref{eq:ab_expand},
\eqref{eq:rbar_bound}, \eqref{eq:rhoJ},
\eqref{eq:Drho_v}, and
$r'(\tilde u_{i+1})\ge1-4\kappa$ for $i\in J$,
we obtain
\[
\begin{aligned}
D_{\tilde{\vct{u}}}
\Psi(\tilde{\vct{u}},\tilde{\vct{a}},\tilde{\vct{b}})
[\vct{v}]
&\le
12\kappa
\left(
\frac{17}{26}
+\frac{\kappa}{13}
+C_\delta
\right)
-\frac{1-4\kappa}{13}
\|\vct{\rho}_J(\tilde{\vct{u}})\|^2
+\frac7{13}\kappa+C_\delta
\\
&\le
12\kappa
\left(
\frac{17}{26}
+\frac{\kappa}{13}
+C_\delta
\right)
+\frac7{13}\kappa+C_\delta
-\frac{1-4\kappa}{13}
\frac{(1-4\kappa)^2}
     {1+(1-4\kappa)^2+4\kappa^2}
\\
&<
-C_\delta,
\end{aligned}
\]
where the last inequality follows from
\eqref{eq:direction_condition}.
On the other hand, \eqref{eq:small_u_grad} and
$\|\vct{v}\|\le1$ imply
\[
\begin{aligned}
\left|
D_{\tilde{\vct{u}}}
\Psi(\tilde{\vct{u}},\tilde{\vct{a}},\tilde{\vct{b}})
[\vct{v}]
\right|
&\le
\|\nabla_{\tilde{\vct{u}}}
\Psi(\tilde{\vct{u}},\tilde{\vct{a}},\tilde{\vct{b}})\|
\,\|\vct{v}\|
\\
&<
C_\delta,
\end{aligned}
\]
a contradiction. Therefore,
\[
\Rcal_{C_0}
(\tilde{\vct{u}},\tilde{\vct{a}},\tilde{\vct{b}})
\ge
C_\delta,
\]
which proves the second claim.

\emph{Part 3: Moreau-envelope obstruction.}
Fix any $\vct{x}=(\vct{u},\vct{a},\vct{b})\in X_0$ with $u_T=0$.

By Lemma~\ref{lem:gap}, $\Phi$ is continuous and bounded below on $X_0$.
Hence, for this $\vct{x}$, the function
\[
 \vct{p}\longmapsto \Phi(\vct{p})+\bar L\norm{\vct{p}-\vct{x}}^2
\]
is coercive on the closed set $X_0$, so
$\Pcal_{1/(2\bar L)}(\vct{x})$ is nonempty.  Fix
$\vct{p}\in\Pcal_{1/(2\bar L)}(\vct{x})$ and set
\begin{equation}\label{eq:gprox}
 \vct{g}:=G_{1/(2\bar L)}(\vct{x};\vct{p})=2\bar L(\vct{x}-\vct{p}).
\end{equation}
By the exact value identity in Lemma~\ref{lem:gap}, $\Phi$ is the
restriction to $X_0$ of the $C^1$ function
$\vct{x}\mapsto \ell_0s^2\Psi(\vct{x}/s)$.  Hence first-order optimality of $\vct{p}$ on
the closed convex set $X_0$ gives
\begin{equation}\label{eq:prox-normal}
 \vct{g}\in\nabla\Phi(\vct{p})+N_{X_0}(\vct{p}).
\end{equation}
Suppose, toward a contradiction, that
\begin{equation}\label{eq:small-prox}
 \norm{\vct{g}}\le\frac{C_\delta}{2}\ell_0s.
\end{equation}
Write
\[
 \vct{p}=s\bar{\vct{p}},
 \qquad
 \bar{\vct{p}}=(\bar{\vct{u}},\bar{\vct{a}},\bar{\vct{b}})\in C_0.
\]
Since $u_T=0$,
\begin{align}
 |\bar u_T|
 &=\frac{|p_{u_T}-x_{u_T}|}{s}
 \le\frac{\norm{\vct{p}-\vct{x}}}s
 =\frac{\norm{\vct{g}}}{2\bar Ls}
 \le\frac{C_\delta\ell_0}{4\bar L}
 =\frac{C_\delta}{4C_{\ell}}
 <\frac14.
 \label{eq:terminal-quarter}
\end{align}
The last inequality follows from the fixed choice
$C_\ell>C_\delta$.

The value identity gives
\begin{equation}\label{eq:gradient-scaling}
 \Phi(s\bar{\vct{p}})=\ell_0s^2\Psi(\bar{\vct{p}}),
 \qquad
 \nabla\Phi(s\bar{\vct{p}})=\ell_0s\nabla\Psi(\bar{\vct{p}}).
\end{equation}
Moreover, $X_0=sC_0$ implies
$N_{X_0}(s\bar{\vct{p}})=N_{C_0}(\bar{\vct{p}})$, and the latter is a cone.  Therefore
\begin{align}
 \dist\bigl(0,\nabla\Phi(\vct{p})+N_{X_0}(\vct{p})\bigr)
 &=\dist\bigl(0,\ell_0s\nabla\Psi(\bar{\vct{p}})+N_{C_0}(\bar{\vct{p}})\bigr)\notag\\
 &=\ell_0s\,
 \dist\bigl(0,\nabla\Psi(\bar{\vct{p}})+N_{C_0}(\bar{\vct{p}})\bigr)\notag\\
 &=\ell_0s\,\Rcal_{C_0}(\bar{\vct{p}}).
 \label{eq:residual-scaling}
\end{align}
By \eqref{eq:terminal-quarter}, part~2 applies to $\bar{\vct{p}}$, and hence
$\Rcal_{C_0}(\bar{\vct{p}})\ge C_\delta$.  Combining this with
\eqref{eq:prox-normal} and \eqref{eq:residual-scaling} yields
\[
 \norm{\vct{g}}
 \ge\dist\bigl(0,\nabla\Phi(\vct{p})+N_{X_0}(\vct{p})\bigr)
 \ge C_\delta\ell_0s,
\]
contradicting \eqref{eq:small-prox}.  Therefore
\[
 \norm{G_{1/(2\bar L)}(\vct{x};\vct{p})}
 >\frac{C_\delta}{2}\ell_0s,
\]
which is exactly \eqref{eq:prox-lower-base}.  This proves part~3 and
completes the lemma.
\end{proof}

\section{Proofs for the Deterministic Lower Bound}\label{app:deterministic}
\subsection{Proof of Theorem \ref{thm:main}}\label{app:deterministic-theorem}
\begin{proof}
We apply the hard instance constructed in Section~\ref{sec:hard}. By
Lemma~\ref{lem:smooth}, its gradient is $C_\ell\ell_0$-Lipschitz.
Moreover, Lemma~\ref{lem:gradient-lower}, with
$\bar L=C_\ell\ell_0$, shows that every zero-respecting output obtained
before the terminal coordinate is activated satisfies
\[
\left\|
G_{1/(2\bar L)}(\vct{x};\vct{p})
\right\|
>
\frac{C_\delta}{2}\ell_0s.
\]
Thus, the smoothness and stationarity requirements are satisfied if
\begin{equation}\label{eq:det-smooth-stat-conditions}
C_\ell\ell_0=L,
\qquad
\frac{C_\delta}{2}\ell_0s=\epsilon.
\end{equation}

It remains to choose the two chain lengths so that the hard instance
belongs to $\mathcal F_{\rm NCC}(L,\Delta_\Phi,D_{\mathcal{Y}})$.
By Lemma~\ref{lem:gap}, the dual maximizer is feasible whenever
\begin{equation}\label{eq:det-dual-condition}
2N^2R^2s^2
\le
\left(\frac{D_{\mathcal{Y}}}{4}\right)^2.
\end{equation}
Under this feasibility condition, Lemma~\ref{lem:gap}
gives the exact initial value gap. We therefore impose the additional requirement
\begin{equation}\label{eq:det-gap-condition}
\Phi(\vct{0})-\inf_{\vct{x}\in X_0}\Phi(\vct{x})
=
c_\eta\left(T+\frac12\right)\ell_0s^2
\le
\Delta_\Phi.
\end{equation}

We now choose the parameters to satisfy
\eqref{eq:det-smooth-stat-conditions}--\eqref{eq:det-gap-condition}.
Set
\[
\ell_0:=\frac{L}{C_\ell},
\qquad
s:=\frac{2\epsilon}{C_\delta\ell_0}
=
\frac{2C_\ell\epsilon}{C_\delta L},
\]
and choose
\[
T
:=
\left\lfloor
\frac{C_\delta^2L\Delta_\Phi}
     {16c_\eta C_\ell\epsilon^2}
\right\rfloor,
\qquad
N
:=
\left\lfloor
\frac{C_\delta LD_{\mathcal{Y}}}
     {16RC_\ell\epsilon}
\right\rfloor,
\qquad
\alpha:=N^{-1/2}.
\]
For a sufficiently small universal constant $c_1$, the assumption
\[
\epsilon
\le
c_1\min\{\sqrt{L\Delta_\Phi},LD_{\mathcal{Y}}\}
\]
ensures that $T,N\ge2$ and
\begin{equation}\label{eq:det-TN-lower}
T-1
\ge
\frac{C_\delta^2L\Delta_\Phi}
     {64c_\eta C_\ell\epsilon^2},
\qquad
N
\ge
\frac{C_\delta LD_{\mathcal{Y}}}
     {32RC_\ell\epsilon}.
\end{equation}
By the choice of $N$,
\[
NRs\le \frac{\Dy}{8},
\]
which implies \eqref{eq:det-dual-condition}.
Moreover, since $T\ge2$, the choice of $T$ gives
\[
c_\eta\left(T+\frac12\right)\ell_0s^2
\le
\frac54c_\eta T\ell_0s^2
\le
\frac{5}{16}\Del
\le \Del.
\]
Thus, \eqref{eq:det-gap-condition} also holds.
Hence, by Lemmas~\ref{lem:gap} and~\ref{lem:smooth},
\[
f\in\mathcal F_{\rm NCC}(L,\Delta_\Phi,D_{\mathcal{Y}}).
\]
Finally, the zero-chain has length $K=1+(T-1)(N+3)
\ge
(T-1)N$. Using \eqref{eq:det-TN-lower},
\[
K
\ge
c_0
\frac{L^2D_{\mathcal{Y}}\Delta_\Phi}{\epsilon^3},
\qquad
c_0
:=
\frac{C_\delta^3}
     {2048c_\eta RC_\ell^2}.
\]
Hence, if
\[
q
<
c_0
\frac{L^2D_{\mathcal{Y}}\Delta_\Phi}{\epsilon^3},
\]
then $q<K$. Lemma~\ref{lem:gradient-lower} therefore gives, for every
$\vct{p}\in\Pcal_{1/(2L)}(\vct{x})$,
\[
\left\|
G_{1/(2L)}(\vct{x};\vct{p})
\right\|
>
\frac{C_\delta}{2}\ell_0s
=
\epsilon.
\]
By Lemma~\ref{lem:weak-convexity},
\[
\left\|
\nabla\varphi_{1/(2L)}(\vct{x})
\right\|
>
\epsilon.
\]
This proves the result.
\end{proof}

\subsection{Proof of Corollary \ref{cor:det-pd-gap}}\label{app:det-pd-gap}
\begin{proof}
Take $c=c_1$ from Theorem~\ref{thm:main} and apply that theorem with the value-gap budget
$\Delta_\Phi=\mathcal G_0$.  We only need to verify that the same
hard instance also satisfies the primal--dual-gap budget. For this construction, $\max_{\vct{y}\in Y_0}f(0,\vct{y})=\Phi_f(\vct{0})=0$.
Write
\[
H(\tilde{\vct{u}})
:=-c_\eta\tilde u_1-c_\eta\sum_{j=2}^T\nu(\tilde u_j)
  +\frac{c_\eta}{2}\sum_{j=1}^T r(\tilde u_j)^2.
\]
At $\vct{y}=0$, the hard instance satisfies
\[
\frac{f(\vct{x},0)}{\ell_0s^2}
=
H(\tilde{\vct{u}})
+\frac12\|\tilde{\vct{a}}-\vct{\rho}(\tilde{\vct{u}})\|^2
+\frac12\|\tilde{\vct{b}}-\bar{\vct{r}}(\tilde{\vct{u}})\|^2
+\frac12\|\tilde{\vct{a}}\|^2
+\left(\frac12-\frac{\alpha^2(N-1)}8\right)\|\tilde{\vct{b}}\|^2.
\]
Since $\alpha^2=1/N$,
\[
\frac12-\frac{\alpha^2(N-1)}8
=\frac{3N+1}{8N}>0.
\]
The lower bound on $H$ established in the proof of Lemma~\ref{lem:gap} gives
\[
H(\tilde{\vct{u}})\ge-c_\eta\left(T + \frac12\right),
\]
and equality is attained at
$\tilde{\vct{u}}^*=(2,1,\ldots,1)$ with $\tilde{\vct{a}}=\tilde{\vct{b}}=0$.
Consequently,
\[
\inf_{\vct{x}\in X_0}f(\vct{x},0)
=-c_\eta\left(T+\frac12\right)\ell_0s^2
=\inf_{\vct{x}\in X_0}\Phi_f(\vct{x}),
\]
where the last equality follows from Lemma~\ref{lem:gap} and $\Phi_f(\vct{0})=0$.  Therefore the
initial primal--dual gap equals the value-function gap:
\[
\max_{\vct{y}\in Y_0}f(0,\vct{y})-\inf_{\vct{x}\in X_0}f(\vct{x},0)
=
\Phi_f(\vct{0})-\inf_{\vct{x}\in X_0}\Phi_f(\vct{x})
\le \mathcal G_0.
\]
Thus the hard instance belongs to
$\mathcal F_{\rm NCC}^{\rm pd}(L,\mathcal G_0,D_{\mathcal{Y}})$.
Applying the query lower bound of Theorem~\ref{thm:main} with
$\Delta_\Phi=\mathcal G_0$ gives the claimed
$\Omega(L^2D_{\mathcal{Y}}\mathcal G_0\epsilon^{-3})$ bound.
\end{proof}
\section{Proofs for the Stochastic Lower Bound}\label{app:stochastic}

\subsection{Proof of Lemma~\ref{lem:stoch-clipped-package}}\label{app:clipped-path}
\begin{proof}
\emph{Part 1: maximizer and optimal value of the clipped dual block.}
Fix $\vct{x}=(\vct{u},\vct{a},\vct{b})\in X_0$
and a block index $i\in\{1,\ldots,T-1\}$.
For this part of the proof, abbreviate $a=a_i$ and $b=b_i$.

Let $\vct{y}^*$ denote the unique maximizer of the deterministic
quadratic block obtained in Lemma~\ref{lem:gap}.
Its explicit form in~\eqref{eq:stoch-ystar} gives
\[
y_k^*-y_{k+1}^*
=
\frac{\alpha b}{2},
\qquad
k=1,\ldots,N-1.
\]
Thus all successive-coordinate differences of $\vct{y}^*$
are equal. For convenience, denote their common value by
\[
d:=\frac{\alpha b}{2}.
\]
Since the primal variables satisfy $|b|\le Rs$ and
$\tau_N=R\alpha s$, we have
\begin{equation}\label{eq:stoch-edge-interior}
|d|
\le
\frac{R\alpha s}{2}
=
\frac{\tau_N}{2}
<
\tau_N.
\end{equation}
Hence all neighboring-coordinate differences of $\vct{y}^*$
lie in the quadratic branch of $\chi_{\tau_N}$.

To verify that $\vct{y}^*$ remains optimal after clipping, define
\begin{equation}\label{eq:stoch-bregman}
D_\tau(d,t)
:=
\chi_\tau(t)-\chi_\tau(d)-d(t-d).
\end{equation}
Because $|d|<\tau$, we have $\chi_\tau'(d)=d$.
By convexity of $\chi_\tau$,
\begin{equation}\label{eq:stoch-bregman-positive}
D_\tau(d,t)\ge0,
\end{equation}
with equality only when $t=d$.
Indeed, $\chi_\tau$ is strictly convex on $(-\tau,\tau)$,
while its two affine pieces have slopes $\pm\tau$,
which are different from $d$.
For convenience, write
\[
e_k(\vct{y}):=y_k-y_{k+1},
\qquad
k=1,\ldots,N-1.
\]
Since $\vct{y}^*$ satisfies
$B_{\alpha,N}\vct{y}^*=\vct{c}(a,b)$ and
$e_k(\vct{y}^*)=d$, the first-order condition of the
deterministic quadratic block gives
\begin{equation}\label{eq:stoch-source-cancel}
\vct{c}(a,b)^\top(\vct{y}-\vct{y}^*)
=
\alpha^2y_1^*(y_1-y_1^*)
+
d\sum_{k=1}^{N-1}
\bigl(e_k(\vct{y})-d\bigr).
\end{equation}
Using this identity to compare the clipped objective at
$\vct{y}$ and $\vct{y}^*$ yields
\begin{align}
h_N^{\rm clip}(a,b;\vct{y}^*)
-
h_N^{\rm clip}(a,b;\vct{y})=
\ell_0\left[
\sum_{k=1}^{N-1}
D_{\tau_N}\bigl(d,e_k(\vct{y})\bigr)
+
\frac{\alpha^2}{2}(y_1-y_1^*)^2
\right].
\label{eq:stoch-clipped-gap}
\end{align}
The right-hand side is nonnegative, and hence $\vct{y}^*$
is a global maximizer.
Moreover, equality in~\eqref{eq:stoch-clipped-gap} requires
\[
y_1=y_1^*,
\qquad
e_k(\vct{y})=d
\quad
\text{for all }k=1,\ldots,N-1.
\]
These relations determine all coordinates recursively,
and therefore $\vct{y}=\vct{y}^*$.
Thus the maximizer is unique.

Finally, \eqref{eq:stoch-edge-interior} shows that every
successive-coordinate difference of $\vct{y}^*$ lies in the
quadratic region of $\chi_{\tau_N}$.
Hence the clipped block and the deterministic quadratic block
agree at $\vct{y}^*$, and Lemma~\ref{lem:gap} gives
\[
h_N^{\rm clip}(a,b;\vct{y}^*)
=
\frac{\ell_0}{2}
\left(a-\frac b2\right)^2.
\]
This proves the first claim.

\emph{Part 2: value function and initial gap.}
By Part~1, the clipped and deterministic dual blocks have
the same unique maximizers and optimal values.
Let $\vct{y}^*(\vct{x})$ denote the concatenation of the
block maximizers $\vct{y}^*(a_i,b_i)$,
$i=1,\ldots,T-1$.
The range estimate in Lemma~\ref{lem:gap} therefore
remains valid:
\[
\|\vct{y}^*(\vct{x})\|^2
\le
2N^2R^2s^2.
\]
Under \eqref{eq:stoch-dual-feas}, the blockwise maximizer
lies in $\operatorname{int}(Y_0)$ and hence is also the
constrained maximizer. Consequently,
\[
\Phi_{\rm sg}(\vct{x})
=
\ell_0s^2
\Psi(\tilde{\vct{u}},\tilde{\vct{a}},\tilde{\vct{b}}),
\]
which is \eqref{eq:stoch-value-identity}.
Since this is exactly the same value function as in the
deterministic construction, the initial-gap identity
\eqref{eq:stoch-gap-exact} follows immediately from
Lemma~\ref{lem:gap}.

\emph{Part 3: dual concavity.}
For each edge, $\vct{y}\mapsto e_k(\vct{y})$ is affine
and $\chi_{\tau_N}$ is convex, hence
$-\chi_{\tau_N}(e_k(\vct{y}))$ is concave.
The anchor $-(\alpha^2/2)y_1^2$ is concave,
the endpoint source is affine, and the $b^2$ correction
is independent of $\vct{y}$.
Thus every clipped path block is concave on the entire
dual space. The outer relay term is independent of
$\vct{y}$, proving part~3.

\emph{Part 4: dimension-free smoothness.}
Fix $\vct{z},\vct{z}'\in X_0\times Y_0$. The primal component is unchanged from the deterministic
construction, so it remains to control the smoothness
of the clipped dual blocks.
Let $E_N:\R^N\to\R^{N-1}$ be the forward difference operator,
\[
(E_N\vct{y})_k=y_k-y_{k+1},
\qquad
k=1,\ldots,N-1.
\]
Since $\|E_N\|_{\rm op}\le2$ and $\chi'_{\tau_N}$
is $1$-Lipschitz, where $\chi'_{\tau_N}$ is applied
coordinatewise, we have
\[
\begin{aligned}
\left\|
E_N^\top\left(
\chi'_{\tau_N}(E_N\vct{y})
-
\chi'_{\tau_N}(E_N\vct{v})
\right)
\right\|
&\le
\|E_N\|_{\rm op}
\left\|
\chi'_{\tau_N}(E_N\vct{y})
-
\chi'_{\tau_N}(E_N\vct{v})
\right\|
\\
&\le
\|E_N\|_{\rm op}^2
\|\vct{y}-\vct{v}\|
\\
&\le
4\|\vct{y}-\vct{v}\|.
\end{aligned}
\]
Thus the clipped squared-difference terms have a Lipschitz
gradient with a constant independent of $N$.

Consider two points $(a,b,\vct{y})$ and
$(a',b',\vct{v})$ in one dual block, and write
\[
d_a:=|a-a'|,
\qquad
d_b:=|b-b'|,
\qquad
d_y:=\|\vct{y}-\vct{v}\|.
\]
From the definition of $h_N^{\rm clip}$,
\[
\nabla_a h_N^{\rm clip}(a,b;\vct{y})
=
\ell_0\alpha y_1,
\]
\[
\nabla_b h_N^{\rm clip}(a,b;\vct{y})
=
\ell_0\left(
-\frac{\alpha}{2}y_N
-\frac{\alpha^2(N-1)}{4}b
\right),
\]
and
\[
\nabla_{\vct{y}}h_N^{\rm clip}(a,b;\vct{y})
=
\ell_0\left[
-E_N^\top\chi'_{\tau_N}(E_N\vct{y})
-\alpha^2y_1\vct{e}_1
+\alpha a\vct{e}_1
-\frac{\alpha}{2}b\vct{e}_N
\right].
\]
Since $\alpha\le1$ and $\alpha^2(N-1)\le1$,
the preceding estimate gives
\[
\frac{1}{\ell_0}
\left|
\nabla_a h_N^{\rm clip}(a,b;\vct{y})
-
\nabla_a h_N^{\rm clip}(a',b';\vct{v})
\right|
\le
d_y,
\]
\[
\frac{1}{\ell_0}
\left|
\nabla_b h_N^{\rm clip}(a,b;\vct{y})
-
\nabla_b h_N^{\rm clip}(a',b';\vct{v})
\right|
\le
\frac12d_y+\frac14d_b,
\]
and
\[
\frac{1}{\ell_0}
\left\|
\nabla_{\vct{y}}h_N^{\rm clip}(a,b;\vct{y})
-
\nabla_{\vct{y}}h_N^{\rm clip}(a',b';\vct{v})
\right\|
\le
d_a+\frac12d_b+5d_y.
\]
Consequently, taking $C_{\rm dual}:=6$, we obtain
\begin{equation}\label{eq:stoch-one-block-smooth}
\left\|
\nabla h_N^{\rm clip}(a,b;\vct{y})
-
\nabla h_N^{\rm clip}(a',b';\vct{v})
\right\|
\le
C_{\rm dual}\ell_0
\sqrt{d_a^2+d_b^2+d_y^2}.
\end{equation}

Let $f_{\rm dual}^{\rm sg}$ denote the sum of the $T-1$
clipped dual blocks.
Since these blocks involve disjoint coordinate groups,
\eqref{eq:stoch-one-block-smooth} yields
\begin{equation}\label{eq:stoch-dual-smooth}
\left\|
\nabla f_{\rm dual}^{\rm sg}(\vct{z})
-
\nabla f_{\rm dual}^{\rm sg}(\vct{z}')
\right\|
\le
C_{\rm dual}\ell_0
\|\vct{z}-\vct{z}'\|.
\end{equation}

Let
$f_{\rm outer}(\vct{x},\vct{y})
:=\ell_0s^2\Psi_0(\vct{x}/s)$.
The primal component is the same as in
Lemma~\ref{lem:smooth}.
Hence, with $C_{\rm primal}:=C_\Psi$ as defined in
Appendix~\ref{app:smooth},
\begin{equation}\label{eq:stoch-outer-smooth}
\left\|
\nabla f_{\rm outer}(\vct{z})
-
\nabla f_{\rm outer}(\vct{z}')
\right\|
\le
C_{\rm primal}\ell_0
\|\vct{z}-\vct{z}'\|.
\end{equation}
Combining \eqref{eq:stoch-dual-smooth} and
\eqref{eq:stoch-outer-smooth} gives
\[
\left\|
\nabla f^{\rm sg}(\vct{z})
-
\nabla f^{\rm sg}(\vct{z}')
\right\|
\le
(C_{\rm primal}+C_{\rm dual})\ell_0
\|\vct{z}-\vct{z}'\|.
\]
By the fixed choice of $C_\ell$ in
Appendix~\ref{app:smooth},
\[
C_{\rm primal}+C_{\rm dual}
=C_\Psi+6\le C_\ell.
\]
Therefore,
\[
\left\|
\nabla f^{\rm sg}(\vct{z})
-
\nabla f^{\rm sg}(\vct{z}')
\right\|
\le
C_\ell\ell_0
\|\vct{z}-\vct{z}'\|.
\]
Since $\ell_0=L/C_\ell$, this further gives
\[
\left\|
\nabla f^{\rm sg}(\vct{z})
-
\nabla f^{\rm sg}(\vct{z}')
\right\|
\le
L\|\vct{z}-\vct{z}'\|,
\]
which proves \eqref{eq:stoch-smooth-package}
and completes part~4.

\emph{Part 5: zero-chain property and dual gradient bound.}
The clipped quadratic coincides with the original quadratic
in a neighborhood of the origin and still involves only
neighboring dual coordinates.
In particular, $\chi'_{\tau_N}(0)=0$, so the same
coordinate-local argument as in
Lemma~\ref{lem:gradient-lower} gives
\[
\prog^{\vct{\pi}}_0
\bigl(\nabla f^{\rm sg}(\vct{z})\bigr)
\le
\min\left\{
\prog^{\vct{\pi}}_0(\vct{z})+1,
K
\right\}.
\]

It remains to prove \eqref{eq:stoch-next-reveal}.
For every dual block $i\in\{1,\ldots,T-1\}$ and every
$j\in\{2,\ldots,N-1\}$,
\[
\nabla_{y_j^{(i)}} f^{\rm sg}(\vct{z})
=
\ell_0
\left[
\chi'_{\tau_N}
\bigl(y_{j-1}^{(i)}-y_j^{(i)}\bigr)
-
\chi'_{\tau_N}
\bigl(y_j^{(i)}-y_{j+1}^{(i)}\bigr)
\right].
\]
Since $|\chi'_{\tau_N}(t)|\le\tau_N,\, \forall t\in\R$, we have
\[
\left|
\nabla_{y_j^{(i)}} f^{\rm sg}(\vct{z})
\right|
\le
2\ell_0\tau_N,
\qquad
j=2,\ldots,N-1.
\]
For $j=N$,
\[
\nabla_{y_N^{(i)}} f^{\rm sg}(\vct{z})
=
\ell_0
\left[
\chi'_{\tau_N}
\bigl(y_{N-1}^{(i)}-y_N^{(i)}\bigr)
-
\frac{\alpha}{2}b_i
\right].
\]
Using $|b_i|\le Rs$ and $\tau_N=R\alpha s$,
\[
\left|
\nabla_{y_N^{(i)}} f^{\rm sg}(\vct{z})
\right|
\le
\ell_0
\left(
\tau_N+\frac{\alpha}{2}|b_i|
\right)
\le
\frac32\ell_0\tau_N
\le
2\ell_0\tau_N.
\]
Therefore, for every feasible $\vct{z}$,
\begin{equation}\label{eq:stoch-dual-uniform}
\left|
\nabla_{y_j^{(i)}} f^{\rm sg}(\vct{z})
\right|
\le
2\ell_0\tau_N,
\qquad
i=1,\ldots,T-1,
\quad
j=2,\ldots,N.
\end{equation}
This proves part~5.
\end{proof}

\subsection{Proof of Lemma~\ref{lem:stoch-expected-obstruction}}\label{app:expected-obstruction}
\begin{proof}
By \eqref{eq:stoch-hidden-terminal}, $\mathbb P(u_T=0)\ge\frac34$. Let
\[
\mathcal E:=\{u_T=0\}.
\]
By Lemma~\ref{lem:stoch-clipped-package}, the clipped construction has the same primal value function as the deterministic construction. Hence the two
instances have the same proximal mapping. Since $C_\ell\ell_0=L$, Part~3 of Lemma~\ref{lem:gradient-lower} and \eqref{eq:stoch-scale} give, on $\mathcal E$,
\[
\left\|
G_{1/(2L)}
\bigl(\vct{x};p_{1/(2L)}(\vct{x})\bigr)
\right\|
>
\frac{C_\delta}{2}\ell_0s
=
\frac43\epsilon.
\]
Therefore,
\[
\begin{aligned}
\mathbb E
\left\|
G_{1/(2L)}
\bigl(\vct{x};p_{1/(2L)}(\vct{x})\bigr)
\right\|
&\ge
\mathbb E\left[
\left\|
G_{1/(2L)}
\bigl(\vct{x};p_{1/(2L)}(\vct{x})\bigr)
\right\|
\mathbf 1_{\mathcal E}
\right] \\
&>
\frac43\epsilon\,\mathbb P(\mathcal E) \\
&\ge
\frac43\epsilon\cdot\frac34
=
\epsilon.
\end{aligned}
\]
This proves the claim.
\end{proof}
\subsection{Proof of Theorem~\ref{thm:stoch-zr}}
\label{app:stochastic-theorem}

\begin{proof}
We apply the clipped hard instance constructed in Section~5.1. Set
\[
    \ell_0=\frac{L}{C_{\ell}},
    \qquad
    s=\frac{8\eps}{3C_\delta\ell_0},
\]
and choose
\[
    T:=
    \left\lfloor
    \frac{9C_\delta^2L\Del}
    {256c_\eta C_{\ell}\eps^2}
    \right\rfloor,
    \qquad
    N:=
    \left\lfloor
    \frac{3C_\delta L\Dy}
    {64RC_{\ell}\eps}
    \right\rfloor,
    \qquad
    \alpha:=N^{-1/2}.
\]
For a sufficiently small universal constant $c_1$, the assumption
\eqref{eq:stoch-smallness} ensures that $T,N\ge4$ and
\begin{equation}\label{eq:stoch-chain-lower}
    T-1
    \ge
    \frac{9C_\delta^2L\Del}
    {1024c_\eta C_{\ell}\eps^2},
    \qquad
    N
    \ge
    \frac{3C_\delta L\Dy}
    {128RC_{\ell}\eps}.
\end{equation}

We first verify that the clipped instance belongs to the required
function class. By the choice of $N$,
\[
    NRs\le\frac{\Dy}{8},
\]
and hence
\[
    2N^2R^2s^2
    \le
    \frac{\Dy^2}{32}
    <
    \left(\frac{\Dy}{4}\right)^2.
\]
Thus the dual-feasibility condition in
Lemma~\ref{lem:stoch-clipped-package} is satisfied. By the choice of $T$,
\[
    c_\eta T\ell_0s^2
    \le
    \frac{\Del}{4}.
\]
Moreover, since the quantity inside the floor defining $T$ is at least
$4$,
\[
    c_\eta\ell_0s^2
    \le
    \frac{\Del}{16}.
\]
Therefore, Lemma~\ref{lem:stoch-clipped-package} gives
\[
\begin{aligned}
    \Phi_{\rm sg}(0)-\inf\Phi_{\rm sg}
    &=
    c_\eta\left(T+\frac12\right)\ell_0s^2 \\
    &\le
    \frac{\Del}{4}
    +
    \frac{\Del}{32} \\
    &=
    \frac{9\Del}{32}
    <
    \Del.
\end{aligned}
\]
Together with the dual concavity, $L$-smoothness, and
$\operatorname{diam}(Y_0)=\Dy$, this shows that $f^{\rm sg}\in\mathcal F_{\rm NCC}(L,\Del,\Dy)$. The oracle constructed in Section~\ref{subsec:stoch-oracle}, with
$p_N$ chosen as in \eqref{eq:stoch-pN}, is unbiased and has
mean-square error at most $\sigma^2$.

We next estimate the number of oracle calls required to cross the
randomized dual coordinates. Recall that
$M=(T-1)(N-1)$. Since $N\ge4$, we have $N-1\ge N/2$. Combining this with
\eqref{eq:stoch-chain-lower} gives
\begin{equation}\label{eq:stoch-M-lower}
    M
    \ge
    c_M
    \frac{L^2\Dy\Del}{\eps^3},
    \qquad
    c_M
    :=
    \frac{27C_\delta^3}
    {262144\,c_\eta R C_{\ell}^2}.
\end{equation}
We also have
\begin{equation}\label{eq:stoch-N-lower}
    N
    \ge
    c_N\frac{L\Dy}{\eps},
    \qquad
    c_N:=
    \frac{3C_\delta}{128RC_{\ell}}.
\end{equation}

On the other hand, since
\[
    G_N=2\ell_0\tau_N,
    \qquad
    \tau_N=R\alpha s,
    \qquad
    \alpha^2N=1,
\]
we have
\begin{equation}\label{eq:stoch-G-identity}
    G_N^2N
    =
    4R^2\ell_0^2s^2
    =
    c_G\eps^2,
    \qquad
    c_G:=
    \frac{256R^2}{9C_\delta^2}.
\end{equation}
The choice \eqref{eq:stoch-pN} gives
\[
    p_N^{-1}
    =
    \max\left\{
    1,\frac{\sigma^2}{G_N^2}
    \right\}
    \ge
    \frac12
    \left(
    1+\frac{\sigma^2}{G_N^2}
    \right).
\]
Consequently,
\[
\begin{aligned}
    \frac{M}{p_N}
    \ge
    \frac{M}{2}
    +
    \frac{M\sigma^2}{2G_N^2}=
    \frac{M}{2}
    +
    \frac{MN\sigma^2}{2G_N^2N}.
\end{aligned}
\]
Using \eqref{eq:stoch-M-lower},
\eqref{eq:stoch-N-lower}, and
\eqref{eq:stoch-G-identity}, we obtain
\begin{equation}\label{eq:stoch-M-over-p}
    \frac{M}{p_N}
    \ge
    c_{\rm d}'
    L^2\Dy\Del\,\eps^{-3}
    +
    c_{\rm n}'
    L^3\Dy^2\Del\sigma^2\,\eps^{-6},
\end{equation}
where $c_{\rm d}':=\frac{c_M}{2}, \,\, c_{\rm n}':=\frac{c_Mc_N}{2c_G}$
are universal positive constants. Since $T,N\ge4$, we have $M\ge9$, and therefore
\[
    \frac{M}{4}
    \le
    \frac{M-\log 4}{2}.
\]
Choose
$c_{\rm d}:=\frac{c_{\rm d}'}{4}, \,\,    c_{\rm n}:=\frac{c_{\rm n}'}{4}$. Then every $q$ satisfying \eqref{eq:stoch-query-lower} obeys
\[
\begin{aligned}
    q
    &<
    \frac14
    \left(
    c_{\rm d}'
    L^2\Dy\Del\,\eps^{-3}
    +
    c_{\rm n}'
    L^3\Dy^2\Del\sigma^2\,\eps^{-6}
    \right) \\
    &\le
    \frac{M}{4p_N} \\
    &\le
    \frac{M-\log 4}{2p_N}.
\end{aligned}
\]
Thus the progress condition
\eqref{eq:stoch-progress-budget} is satisfied.

Lemma~\ref{lem:stoch-expected-obstruction} therefore yields, for the
unique proximal point
\[
    \vct{p}=p_{1/(2L)}(\vct{x}),
\]
that
\[
    \mathbb E
    \norm{G_{1/(2L)}(\vct{x};\vct{p})}
    >
    \eps.
\]
By Lemma~\ref{lem:weak-convexity},
\[
    G_{1/(2L)}
    \bigl(\vct{x};p_{1/(2L)}(\vct{x})\bigr)
    =
    \nabla(\varphi_{\rm sg})_{1/(2L)}(\vct{x}),
\]
and hence
\[
    \mathbb E
    \norm{
    \nabla(\varphi_{\rm sg})_{1/(2L)}(\vct{x})
    }
    >
    \eps.
\]
This proves \eqref{eq:stoch-main-obstruction}. Consequently, any
stochastic zero-respecting algorithm that returns an expected
$\eps$-stationary point requires
\[
    \Omega\!\left(
    L^2\Dy\Del\,\eps^{-3}
    +
    L^3\Dy^2\Del\sigma^2\,\eps^{-6}
    \right)
\]
oracle calls.
\end{proof}

\subsection{Proof of Corollary \ref{cor:g0-lower-bound}}\label{app:stoch-pd-gap}

\begin{proof}
Consider the stochastic hard instance used in
Theorem~\ref{thm:stoch-zr}. By
Lemma~\ref{lem:stoch-clipped-package}, the clipped construction has
the same primal value function as the deterministic construction:
\[
    \Phi_{\rm sg}(\vct{x})
    =
    \Phi_f(\vct{x}).
\]
Moreover, since $\chi_{\tau_N}(0)=0$, clipping does not change the
objective at $\vct{y}=\vct{0}$:
\[
    f^{\rm sg}(\vct{x},\vct{0})
    =
    f(\vct{x},\vct{0}).
\]
Therefore the argument in the proof of
Corollary~\ref{cor:det-pd-gap} applies verbatim and gives
\[
\begin{aligned}
    \max_{\vct{y}\in Y_0}
    f^{\rm sg}(\vct{0},\vct{y})
    -
    \inf_{\vct{x}\in X_0}
    f^{\rm sg}(\vct{x},\vct{0})
    &=
    \Phi_{\rm sg}(\vct{0})
    -
    \inf_{\vct{x}\in X_0}
    \Phi_{\rm sg}(\vct{x}).
\end{aligned}
\]
Now apply Theorem~\ref{thm:stoch-zr} with
$\Del=\mathcal G_0$. The common gap is at most $\mathcal G_0$, so the
resulting instance belongs to
$\mathcal F_{\rm NCC}^{\rm pd}(L,\mathcal G_0,\Dy)$.
The claimed oracle lower bound follows directly.
\end{proof}

\end{document}